\documentclass[11pt,usenames,dvipsnames]{amsart}
\usepackage[utf8]{inputenc}
\usepackage[a4paper,top=2.8cm,bottom=2cm,left=2.5cm,right=2.5cm,marginparwidth=1.75cm]{geometry}

\usepackage[svgnames]{xcolor}

\usepackage{amsfonts,bbm,mathrsfs}
\usepackage{amsmath}
\usepackage{amsthm}
\usepackage{amssymb}

\usepackage{subcaption}

\usepackage[all]{xy}
\usepackage{graphicx}
\usepackage{color,mathtools}

\usepackage[T1]{fontenc}
\usepackage{multirow}

\usepackage{enumerate}

\usepackage{booktabs}

\usepackage[style=alphabetic,doi=false,isbn=false,url=false,eprint=true,
backend=biber,giveninits=true,maxalphanames=5,maxbibnames=8,hyperref]{biblatex}
\usepackage{amsmath}
\usepackage{amssymb}
\usepackage{amsthm}
         
\usepackage{tikz,tikz-cd}
\usetikzlibrary{decorations.pathreplacing}
\usetikzlibrary{decorations.markings}
\usetikzlibrary{calc,cd}

\definecolor{maroon(html/css)}{rgb}{0.5, 0.0, 0.0}\definecolor{cobalt}{rgb}{0.0, 0.28, 0.67}
\usepackage[bookmarksnumbered=true]{hyperref} 
\hypersetup{
	colorlinks = true,
	linkcolor = cobalt,
	anchorcolor = cobalt,
	citecolor = cobalt,
	filecolor = cobalt,
	urlcolor = cobalt
}

\renewcommand{\epsilon}{\varepsilon}

\newcommand{\R}{{\mathbb R}}

\newcommand{\tc}{\operatorname{tc}}
\newcommand{\osp}{\operatorname{OSP}}

\DeclareMathOperator{\op}{op}

\newcommand{\cH}{{\mathcal H}}

\newcommand{\cF}{{\mathcal F}}

\newcommand{\T}{{\bf T}}

\def \R{\mathbb{R}}

\def \/{\big/}

\def \S{\mathbf{S}}

\def \Span{\textup{span}}

\newcommand{\upto}{\raisebox{1.2 mm}{$\curvearrowright$}}
\newcommand{\downto}{\mathbin{\rotatebox[origin=c]{-180}{$\curvearrowleft$}}}

\usepackage{thmtools}
\declaretheorem[numberwithin=section]{theorem}
\declaretheorem[sibling=theorem]{question}

\declaretheorem[sibling=theorem]{proposition}
\declaretheorem[sibling=theorem]{lemma}
\declaretheorem[sibling=theorem]{corollary}
\declaretheorem[sibling=theorem]{problem}
\declaretheorem[sibling=theorem]{definition}

\declaretheorem[sibling=theorem, style=definition]{example}
\declaretheorem[sibling=theorem, style=definition]{remark}

\theoremstyle{plain}

\theoremstyle{definition}

\newtheorem*{Theorem*}{Theorem}
\newtheorem*{Corollary*}{Corollary}

\newenvironment{varthm*}[1]{\begin{list}{}{\labelwidth=0cm
			\leftmargin=0cm
			\listparindent=\parindent}\item[\hspace{\labelsep}\bfseries
		#1.]\itshape}{\end{list}}

\usepackage{cleveref}

\title{When Alcoved Polytopes Add}


\author{Nick Early}
\address{Nick Early, Institute for Advanced Study, USA}
\email{\href{mailto:earlnick@ias.edu}{earlnick@ias.edu}}

\author{Lukas K\"uhne}
\address{Lukas K\"uhne, Universit\"at Bielefeld, Germany and Institute for Advanced Study, USA}
\email{\href{mailto:lkuehne@math.uni-bielefeld.de}{lkuehne@math.uni-bielefeld.de}}

\author{Leonid Monin}
\address{Leonid Monin, École Polytechnique Fédérale de Lausanne (EPFL), Switzerland}
\email{\href{mailto:leonid.monin@epfl.ch}{leonid.monin@epfl.ch}}

\begin{document}
	
\begin{abstract}
Alcoved polytopes are characterized by the property that all facet normal directions are parallel to the roots $e_i-e_j$.
Unlike other prominent families of polytopes, like generalized permutohedra, alcoved polytopes are not closed under Minkowski sums. 
We nonetheless show that the Minkowski sum of a collection of alcoved polytopes is alcoved if and only if each pairwise sum is alcoved.
This implies that the type fan of alcoved polytopes is determined by its two-dimensional cones.  
Moreover, we provide a complete characterization of when the Minkowski sum of alcoved simplices is again alcoved via a graphical criterion on pairs of ordered set partitions. 
Our characterization reduces to checking conditions on restricted partitions of length at most six. In particular, we show how the Minkowski sum decompositions of the two most well-known families of alcoved polytopes, the associahedron and the cyclohedron, fit in our framework.  Additionally, inspired by the physical construction of one-loop scattering amplitudes, we present a new infinite family of alcoved polytopes, called $\hat{D}_n$ polytopes.
We conclude by drawing a connection to matroidal blade arrangements and the Dressian. 

\end{abstract}

\maketitle

	\begingroup
	\let\cleardoublepage\relax
	\let\clearpage\relax
	\tableofcontents
	\endgroup

\section{Introduction}
A polytope in $\cH_n=\{x_1+\dots + x_n=0\}\subset \R^n$  is \emph{alcoved} if all its facet normals are parallel to the roots $e_i-e_j$ for some $i\neq j\in \left[n\right]$.
Equivalently, a polytope is alcoved if it is determined by the parameters  $a_{i,j}\in \R$ for $1\le i,j,\le n$ via the equation $x_1+\dots +x_n=0$ and  the inequalities
\begin{equation}\label{eq:alcoved}
	x_i-x_j\le a_{i,j} \text{ for all }i,j\in \lbrack n \rbrack, i\neq j.
\end{equation}

Alcoved polytopes were introduced by Lam and Postnikov~\cite{lam2007alcoved} and appeared in different fields under different names.
They are known in the literature as \emph{polytropes} as they are tropical polytopes which are convex in the usual sense~\cite{JK10}.
Moreover, they are \emph{Lipschitz polytopes} (for non-symmetric finite metric spaces)~\cite{GP17,lipschitzpolytope}.
The class of alcoved polytopes includes order polytopes, hypersimplices, and the associahedron. In applications, alcoved polytopes play a key role in phylogenetics \cite{yoshida2019tropical}, mechanism design \cite{crowell2016tropical}, algebraic statistics \cite{Wasserstein}, scattering amplitudes  \cite{CEGM2019,Ear22}, positive configuration spaces \cite{arkani2021positive} and amplituhedra \cite{parisi2021m}, and building theory~\cite{JSY07}. Unbounded alcoved polyhedra were also considered under the names of shortest path polyhedra in \cite[Section 8.3]{schrijver2003combinatorial} or \cite{khachiyan2008generating}  and of weighted digraph polyhedra in \cite{joswig2016weighted}.

Unlike for instance generalized permutohedra, alcoved polytopes are not closed under Minkowski sums in general.
This naturally raises the question when alcoved polytopes add.
\begin{problem}\label{prob:sum_alcoved_polytopes}Let $P,Q\subseteq \cH_n$ be alcoved polytopes.
	When is the Minkowski sum $P+Q$ alcoved?

    We call the alcoved polytopes $P$ and $Q$ \emph{compatible} if their sum $P+Q$ is alcoved.
\end{problem}

Problem \ref{prob:sum_alcoved_polytopes} is intimately tied to classification of normal fans of alcoved polytopes and the general study of their \emph{type fan} following McMullen~\cite{McMullen}.
Suppose the parameters $a_{i,j}$ from~\eqref{eq:alcoved} minimally define an alcoved polytope.
They then satisfy the following triangle inequalities \cite{joswig2016weighted}:
\[
a_{i,j}+a_{j,k}\geq a_{i,k}, \text{ for all } i,j,k.
\]
The cone in $\mathbb{R}^{(n-1)n}$ defined by these inequalities has an internal fan structure given by the different normal fans of alcoved polytopes.
This fan is the \emph{type fan of alcoved polytopes}~$\cF_n$.
This fan was also studied in the context of tropical geometry and optimization in~\cite{joswig2022parametric,tran2017enumerating}.
It is furthermore closely related to the so-called \emph{resonance  arrangement} studied in \cite{early2018honeycomb,kuehne23}, see~\Cref{sec:all_subset}.
In this setting, two polytopes $P$ and $Q$ are compatible if and only if their corresponding points are part of one potentially larger cone in the fan $\cF_n$, see~\Cref{prop:compatibility_type_fan}.
Understanding the compatibility of alcoved polytopes is therefore equivalent to the study of the cone structure of the type fan of alcoved polytopes  $\cF_n$.

Problem \ref{prob:sum_alcoved_polytopes} reveals a rich combinatorial structure even if we restrict our attention to the subclass of alcoved simplices. In particular, several prominent alcoved polytopes such as the associahedron and the cyclohedron are Minkowski sums of alcoved simplices, see~\Cref{sec:examples}. Motivated by this, we are particularly interested in the question of compatibility of alcoved simplices which we answer completely. 

\subsection{Motivation from physics} Apart from interest from combinatorics, the study of alcoved polytopes and their Minkowski decompositions is motivated from recent geometric approaches to the study of scattering amplitudes in theoretical physics.
The main property making alcoved polytopes encode information about scattering amplitudes is that their faces exhibit factorization properties reminiscent of physical singularities.

In particular, the associahedron, which maps onto the connected components in the tiling of the configuration space $\mathcal{M}_{0,n}$, has appeared in string theory for decades. It appears as a natural compactification of configuration spaces appearing in the Koba-Nielsen formulation of string theory amplitudes.
More recently in current approaches to quantum field theory \cite{CHY}; it governs the singularity locus of the $\text{tr}(\phi^3)$ amplitude.
Moreover this amplitude satisfies certain important physical compatibility constraints on pairs of poles, called the Steinmann relations.  
The study of pairwise compatible collections of alcoved polytopes could therefore generalize the role of the associahedron in the Koba-Nielsen string integral and in the Cachazo-He-Yuan formalism~\cite{CHY} to arbitrary alcoved polytopes, suggesting new connections between polytope theory and scattering amplitudes.   
A more recent such instance is the so-called $\hat{D}$-polytope recently described in the physics literature~\cite{dhat}, see Theorem \ref{thm: assocCycDHat}.

These results are examples of more general interplays of combinatorics, geometry and theoretical physics which recently have led to remarkable joint developments. 
 For instance, the study of generalized biadjoint scalar amplitudes \cite{CEGM2019} has revealed deep connections between tropical geometry, cluster algebras, quantum affine algebras \cite{early2024tropical}, and the geometry of the positive and chirotropical Grassmannians \cite{cachazo2024color,antolini2024chirotropical}. These connections suggest that broader families of polytopes could encode fundamental physical principles.

Let us finish this subsection with two concrete questions motivated from physics on alcoved polytopes which we will address in future works.
\begin{itemize}
    \item {\bf Construction of binary geometries from families of alcoved polytopes.} Binary geometries are affine varieties with stratifications determined by certain simplicial complexes.
    Classical examples of binary geometry are the associahedra \cite{arkani2021stringy} and a more recent one are the pellytopes \cite{pellytopes}.
    To see more binary geometries constructed through alcoved polytopes, it is crucial to compute the so-called $u$-variables arising from their Minkowski decompositions.   
\item{\bf Defining alcoved amplitudes.} Certain scattering amplitudes may be presented as $\varepsilon\to 0$ limits of integrals of the following form  called \textit{stringy} integrals:
\[
\int_{\mathbb{R}^d_{>0}}\frac{dy}{y}\prod_{j=1}^d x_j^{s_j}\prod_f f(y)^{\varepsilon s_f},
\]
where $f(y)$ are some given irreducible polynomials, the $s_j$ and $s_f$ are real variables that satisfy the requirement that the origin in $\mathbb{R}^d$ is in the interior of the Newton polytope of the product.
In the case when the Minkowski sum of the Newton polytopes of $f$ is the ABHY associahedron in kinematic space, it produces the classical Koba-Nielsen string integral.
We would like to understand more cases when the Newton polytope above is an alcoved polytope.
\end{itemize}

\subsection{Results}
We now describe the main contributions and the structure of this article.
In~\Cref{sec:type_fan} we prove that the compatibility of alcoved polytopes can be checked on pairs:
\begin{varthm*}{Theorem A}\label{thm:A}
	Let $P_1,\dots,P_k$ be alcoved polytopes in $\cH_n$.
	Suppose $P_i$ and $P_j$ are pairwise compatible for all $i\neq j\in\left[n\right]$.
	Then the entire collection is compatible, i.e., $P_1+\dots+P_k$ is alcoved.
\end{varthm*}
This in particular means that the combinatorial structure of the type fan is completely determined by its $2$-dimensional cones, see~\Cref{thm:2_determined}.

After discussing the intersection of root cones in~\Cref{sec:root_cones}, we focus on alcoved simplices in~\Cref{sec:alcoved_simplices}.
As they are Minkowski indecomposable, they are among the rays of the type fan~$\cF_n$.
Up to translation and scaling, every alcoved simplex in~$\cH_n$ is characterized by an \emph{ordered set partition} of $\left[n\right]$, see~\Cref{prop:every_simplex}.

In~\Cref{sec:compatibility} we give a characterization for the compatibility of alcoved simplices.
\begin{varthm*}{Theorem B}\label{thm:B}
    Let $\S$ and $\T$ be two ordered set partitions of $[n]$ corresponding to the alcoved simplices $\Delta_\S$ and $\Delta_\T$ in $\cH_n$.
	The simplices $\Delta_\S$ and $\Delta_\T$ are compatible if and only if the simplices corresponding to the restricted partitions $\S|_I$ and $\T|_I$ are compatible for all  $I\subset [n]$ with $|I|\le 6$.
\end{varthm*}

Equivalently, this theorem says that two alcoved simplices are compatible if and only if their faces of dimension at most five are pairwise compatible.
Full-dimensional alcoved simplices are parametrized by cyclic permutations.
In this case, the theorem says that two full-dimensional alcoved simplices are compatible if and only if their permutations avoid three patterns, one of length four and two of length six, see~\Cref{rmk:pattern_avoidance} for details.

We prove this theorem in two steps.
First, we construct a graph associated to the simplices  $\Delta_\S$ and $\Delta_\T$ of the ordered set partitions $\S$ and $\T$ such that the simplices are compatible if and only if there is no cycle in this graph of a specific type.
Subsequently, we show that if the graph has such a cycle we can already find such a cycle on a subset of $\left[n\right]$ of size at most~$6$.

Finally, we would like to emphasize, that the reduction of compatibility of alcoved simplices to compatibility of restricted ordered set partitions in Theorem~\hyperref[thm:B]{B} is quite special.
For instance, the analogous question of when the intersection of two root subspaces, i.e., subspaces generated by roots, is again a root subspace does not satisfy such a reduction property, see Remark~\ref{rmk:non_reduction}.

In~\Cref{sec:examples} we discuss a number of prominent examples of alcoved polytopes and amongst others confirm that the $\hat{D}_n$-polytope is alcoved.
Lastly, we connect our results to \emph{matroidal blade arrangements} in~\Cref{sec:blades}.
These are arrangements of normal fans of alcoved simplices placed on the vertices of a hypersimplex.
The first author completely characterized when the normal fan of one fixed alcoved simplex at different vertices induces a matroidal subdivision of the hypersimplex.
Our characterization of compatible alcoved simplices complements these results as they characterize the matroidal subdivisions induced by different normal fans placed at one fixed vertex.
The common generalization of placing general blades at different vertices remains an exciting question for future research.

\section{The type fan of alcoved polytopes}\label{sec:type_fan}
In this section we give a proof of Theorem~\hyperref[thm:A]{A} and discuss its implications for the type fan of alcoved polytopes.
Let us start with the definitions.

\begin{definition}
    We call the vectors $e_{ij}:= e_i - e_j \in \cH_n \subset \R^n$ the roots of type $A$. 
    We say that a vector subspace $L\subset \cH_n$ is a \emph{root subspace} if it is spanned by roots.
\end{definition}

We start with a reformulation of alcoved polytopes in terms of their normal fans.
We denote the normal fan of a polytope $P$ by $\Sigma_P$.
A polytope~$P$ in the hyperplane $\cH_n$ is alcoved if the lineality space $L$ of the normal fan $\Sigma_P$ is a root subspace and the rays of the quotient fan $\Sigma_P/L$ are generated by roots.
This follows immediately from the definition of alcoved polytopes in~\eqref{eq:alcoved}.

The main technical result we need  to prove Theorem~\hyperref[thm:A]{A} is the following proposition. 

\begin{proposition}\label{prop:alcovedspaces}
    Let $L_1, \dots, L_k\subset \cH_n$ be root subspaces, such that $L_s\cap L_t$ is a root subspace for all $1\leq s,t\leq k$. Then $L_1\cap\dots\cap L_k$ is a root subspace.
\end{proposition}

To prove Proposition~\ref{prop:alcovedspaces}, we introduce a graph $\Gamma_L$ which completely determines the root subspace $L$ (Definition~\ref{def:GL}).
Subsequently, we give a graphical criterion for the compatibility of root subspaces (Lemma~\ref{lem:cycles}).
However, let us first present a proof of Theorem~\hyperref[thm:A]{A} assuming Proposition~\ref{prop:alcovedspaces}.

\begin{proof}[Proof of Theorem~{\hyperref[thm:A]{A}}]
    The goal is to show that $P=P_1+\ldots+P_k$ is alcoved.
    We can assume that $P$ is full-dimensional. Indeed, the lineality space of $L_P$ of the normal fan $\Sigma_P$ of $P$ is the intersection of lineality spaces $L_{P_i}$ of normal fans $\Sigma_{P_i}$ of $P_i$. By the assumption, each pairwise intersection $L_{P_i}\cap L_{P_j}$ is a root space and thus $L_P$ is a root space, by  Proposition~\ref{prop:alcovedspaces}, and we can pass to the quotient of $\Sigma_P/L_P$ in our analysis.
    
    Let $F$ be a facet of $P$. 
    Hence, we can choose faces $F_s$ of $P_s$ for all $1\le s \le k$ such that $F=F_1+\ldots+F_k$.
    Therefore, the normal ray $\rho_F$ is the intersection of the normal cones $\sigma_1\cap\dots\cap\sigma_k$ where $\sigma_s$ is the normal cone of $F_s$ for all $1\le s \le k$.

    The ray $\rho_F$ is generated by a root if and only if the linear span of $\rho_F$ is generated by a root.
    Therefore we aim to work with linear spaces instead of cones.
    To this end we replace each $\sigma_s$ with the smallest face of $\sigma_s$ containing $\rho_F$.
    Thus, we obtain the equality
    \[
   \Span(\rho_F)= \bigcap_{s=1}^k \Span(\sigma_{s}).
    \]
     So it is enough to show that $\bigcap_{s=1}^k \Span(\sigma_{s})$ is a root subspace.

    But by the assumptions of the theorem we have that $\Span(\sigma_{s})$ as well as $\Span(\sigma_{s})\cap\Span(\sigma_{t})$ is generated by roots for every $1\le s,t,\le k$.
    Indeed, since $P_s$ and $P_s+P_t$ are alcoved, the rays of their normal fans are all in  root directions and hence, every cone (and its linear span) is generated by roots.
    Therefore, by Proposition~\ref{prop:alcovedspaces} we get that $ \bigcap_{s=1}^k \Span(\sigma_{s})$ is generated by roots, and thus $\rho_F$ is also a multiple of a root.
\end{proof}

To prove Proposition~\ref{prop:alcovedspaces} we first need to set up some notation.
\begin{definition}\label{def:GL}
    For a root subspace $L\subset \cH_n$ we define an undirected graph $\Gamma_L$ as follows
    \begin{itemize}
        \item[(1)] The vertices of $\Gamma_L$ are labeled by $\left[n\right]$.
        \item[(2)] The graph $\Gamma_L$ has an edge $\{i,j\}$ if $e_{ij}\in L$, note that as $L$ is a subspace, it contains the root $e_{ij}$ if and only if it contains the root $e_{ji}$.
    \end{itemize}
\end{definition}

Since $\cH_n =\{x_1+\dots+x_n=0\}$ is the root subspace of all roots, the graph $\Gamma_{\cH_n}$ is the complete graph on $n$ vertices.

\begin{lemma}\label{lem:cliques}
	Each connected component of $\Gamma_L$ is a complete graph.
\end{lemma}
\begin{proof}
	Indeed let two vertices $s$ and $t$ be in the same connected component of $\Gamma_L$, we will show that there is an edge $\{s,t\}$ in $\Gamma_L$.
	
	Let $\{s=i_1,i_2\},\dots, \{i_{k-1},i_{k}=t\}$ be a path in $\Gamma_L$ connecting $s$ and $t$. By definition of $\Gamma_L$ this implies that $e_{i_ji_{j+1}}\in L$ for all $1\le j \le k-1$. But
	\[
	e_{i_1i_2}+e_{i_2i_3} + \ldots +  e_{i_{k-1}i_k}= e_{i_1i_k}=e_{st} \in L,
	\]
	so $\{s,t\}$ is an edge of $\Gamma_L$.
\end{proof}

Any linear combination of roots $\sum_{1\le i<j\le n} w_{ij}e_{ij}$ defines a weighted sum of oriented edges of $\Gamma_{\cH_n}$ by assigning the orientation $j\to i$ and weight $w_{ij}$ to every edge $\{i,j\}$.
Two such combinations $\sum_{1\le i<j\le n} w_{ij}e_{ij}$ and $\sum_{1\le i<j\le n} w'_{ij}e_{ij}$ give rise to the same element of~$\cH_n$ if and only if their difference $\sum_{1\le i<j\le n} (w_{ij}-w'_{ij})e_{ij}$ is zero.
This linear combination is zero if and only if it defines a linear combination of oriented cycles in $G_{\cH_n}$ where we view the edges with a negative weight as oriented as $i\to j$.

Now, let $L, M$ be two root subspaces within $\cH_n$.
By the above observation, an element $x$ is in the intersection $L\cap M$ if it can be simultaneously represented by a linear combination of oriented edges from $\Gamma_L$ and $\Gamma_M$.
The difference of these two representation is thus a linear combination of cycles supported on the edges $\Gamma_L\cup \Gamma_M$.

On the other hand, every oriented cycle $C$ supported on  $\Gamma_L\cup \Gamma_M$ yields an element of  $L\cap M$ by taking the linear combination 
\begin{equation}\label{eq:xc}
x_C=\sum_{\{i,j\}\in C\cap \Gamma_L}e_{ij}=-\sum_{\{s,t\}\in C\cap \Gamma_M}e_{st}.
\end{equation}
This leads to the following lemma:

\begin{lemma}\label{lem:spaceintersect}
    Let $L, M$ be two root subspaces within $\cH_n$.
    The intersection $L\cap M$ is generated by elements corresponding to oriented cycles in $\Gamma_L\cup \Gamma_M$ as in~\eqref{eq:xc}.

    In particular, if $L\cap M$ is a root subspace, we have $\Gamma_{L\cap M}=\Gamma_L\cap \Gamma_M$.
\end{lemma}

In addition, a more careful analysis yields the following lemma.
Recall that a chord $c$ in a cycle $C$ is an edge connecting two vertices that are not adjacent in $C$.

\begin{lemma}\label{lem:cycles}
    Let $L, M$ be two root subspaces within $\cH_n$.
    The intersection $L\cap M$ is a root subspace if and only if every cycle of length at least four in $\Gamma_L\cup \Gamma_M$ has a chord.
\end{lemma}
\begin{proof}
	First suppose that $L\cap M$ is a root subspace.
	Let $C$ be a cycle of length at least four in $\Gamma_L\cup \Gamma_M$.
	If $C$ contains two consecutive edges of $\Gamma_L$ or of $\Gamma_M$ it contains a chord as the connected components of both graphs are complete graphs by~\Cref{lem:cliques}.
	Hence, we can assume that $C$ is strictly alternating, i.e., if an edge in $C$ is in $\Gamma_L$ the next one on $C$ must be in $\Gamma_M$ and not in $\Gamma_L$ and vice versa.
	By~\Cref{lem:spaceintersect} the intersection $L\cap M$ is generated by oriented cycles.
	As $L\cap M$ is assumed to be a root subspace, the element $x_C$ corresponding to the cycle $C$  can be written as a sum of cycles in $\Gamma_L\cup \Gamma_M$ each of length two as cycles of length two correspond to roots.
	This means that there is an edge $e\in C$ that appears both in $\Gamma_L$ and $\Gamma_M$.
	This however contradicts the assumption that $C$ is strictly alternating and thus every such cycle must have a chord.
	
	For the converse, assume that every cycle of length at least four in  $\Gamma_L\cup \Gamma_M$ has a chord.
	Choose an arbitrary element $x\in L\cap M$.
	By Lemma~\ref{lem:spaceintersect} $x$ is a linear combination of $x_{C_1},\dots, x_{C_k}$ for oriented cycles $C_1,\dots,C_k$ in $\Gamma_L\cup \Gamma_M$.
	Moreover as above, these cycles can be chosen to be strictly alternating.
	Suppose a cycle $C_i$ in this collection is of length at least four.
	By our assumption, $C_i$ thus has a chord.
	Splitting up the cycle $C_i$ along this chord into the two shorter cycles $C_i'$ and $C_i''$, we can write $x_{C_i}=x_{C_i'}+x_{C_i''}$.
	Iterating this process means that we can shorten the cycles whenever they are of length at least four.
	So in total we can assume that all cycles $C_1,\dots,C_k$ are of length two.
	As the point corresponding to a cycle of length two is a root this implies that $x$ is generated by roots.
	Hence, $L\cap M$ is a root subspace.
\end{proof}

We give one final technical lemma that we will use in the proof of Proposition~\ref{prop:alcovedspaces}.

\begin{lemma}\label{lem:complete_grpah}
	Suppose that $L,M\subseteq \cH_n$ are two root subspaces such that $L\cap M$ is a root subspace, too.
	Assume that there is a strictly alternating cycle $C$ in $\Gamma_L\cup \Gamma_M$ such that every edge in $C\cap \Gamma_M$ belongs to a different connected component of $\Gamma_M$.
	This implies that the restriction of $\Gamma_{L}$ to the vertices of $C$ is a complete graph.
\end{lemma}
\begin{proof}
	For a contradiction, assume that there is a strictly alternating cycle $C$ in $\Gamma_L\cup \Gamma_M$ such that edges in $C\cup \Gamma_M$ belong to different connected components of $\Gamma_M$ and the restriction of  $\Gamma_L$ to the vertices of $C$ is not a complete graph.
	Suppose that $C$ is such a cycle of minimal length.
	
	As a cycle of length two is clearly complete on its vertices, $C$ must be of length at least four.
	By~\Cref{lem:cycles} and the assumption that $L\cap M$ is a root subspace, the cycle $C$ must have a chord~$e$ in $\Gamma_L\cup \Gamma_M$.
	
	As the edges in $C\cap \Gamma_M$ all belong to pairwise different connected components of  $\Gamma_M$, the chord $e$ cannot belong to $\Gamma_M$ but must be in $\Gamma_L$.
	If the cycle $C$ is of length four the chord $e$ thus connects the edges of $C\cap \Gamma_M$ which together with~\Cref{lem:cliques} implies that $C\cap \Gamma_M$ is a complete graph as claimed.
	
	If the cycle $C$ is of length at least six we can split up the cycle $C$ into two smaller cycles along the chord $e$ one of which is at least of length four and still satisfies the assumption on $\Gamma_M$.
	This contradicts our assumption on $C$ being the cycle of minimal length of this kind.
\end{proof}

Now we are finally ready to prove Proposition~\ref{prop:alcovedspaces}.
\begin{proof}[Proof of Proposition~\ref{prop:alcovedspaces}]
    It is enough to show the statement for $k=3$.
    Indeed assuming this case, the general case follows via induction after replacing $L_1,L_2$ by $L_1\cap L_2$.
    In the case $k=3$, let us define $L_{ij}:= L_i\cap L_j$ for $1\le i<j\le3$.
    Since by assumption $L_{ij}$ is a root subspace, we have $\Gamma_{L_{ij}} = \Gamma_{L_i}\cap \Gamma_{L_j}$ by~\Cref{lem:spaceintersect}.
    
	So our aim is to show that  $L_1\cap L_2 \cap L_3=L_{12}\cap L_3$ is a root subspace.
	To this end, assume that there is a cycle $C$ of length at least 4 in $\Gamma_{L_{12}}\cup \Gamma_{L_3}$.
	By Lemma~\ref{lem:cycles} we need to show that~$C$ has a chord in $\Gamma_{L_{12}}\cup \Gamma_{L_3}$.
	
	Assume that $C$ has no chord in $\Gamma_{L_{12}}\cup \Gamma_{L_3}$.
	Since the connected components of both $\Gamma_{L_{12}}$ and $\Gamma_{L_3}$ are cliques this means that the edges in $C$ alternate between $\Gamma_{L_{12}}$ and $\Gamma_{L_3}$.
	In other words, the cycle $C$ is of the form
	\[
	\{i_1,i_2\},\{i_2,i_3\},\dots,\{i_{2s},i_{2s+1}=i_1\},
	\]
	where $\{i_{2k-1},i_{2k}\}$ is an edge of $\Gamma_{L_3}$ and $\{i_{2k},i_{2k+1}\}$ is an edge of $\Gamma_{L_{12}}$ such that any two edges belong to different components of $\Gamma_{L_3}$  and $\Gamma_{L_{12}}$, respectively.
	
	Every edge in $\Gamma_{L_{12}}\cup \Gamma_{L_3}$ is also contained in $\Gamma_{L_{1}}\cup \Gamma_{L_3}$, thus $C$ is also a cycle in $\Gamma_{L_{1}}\cup \Gamma_{L_3}$.
	The assumption that $L_{13}$ is a root subspace together with~\Cref{lem:complete_grpah} implies that the restriction of $\Gamma_{L_1}$ to the vertices of $C$ is a complete graph.

    The same argument also applies to $\Gamma_{L_2}$ restricted to the vertices of $C$.
    Therefore, $\Gamma_{L_{12}}=\Gamma_{L_1}\cap \Gamma_{L_2}$ restricted to vertices of $C$ is also a complete graph which contradicts the assumption of $C$ being chordless in $\Gamma_{L_{12}}\cup \Gamma_{L_3}$.
\end{proof}

 Note that our proof of Theorem~\hyperref[thm:A]{A} does not use the boundedness of alcoved polytopes and thus is also applicable to the case of unbounded alcoved polyhedra (or weighed digraph polyhedra) which brings us to a more general result
 \begin{theorem}
     Let $P_1,\dots,P_k$ be (possibly unbounded) alcoved polyhedra in $\cH_n$.
	Suppose $P_i$ and $P_j$ are pairwise compatible for all $i\neq j\in\left[n\right]$.
	Then the entire collection is compatible, i.e., $P_1+\dots+P_k$ is alcoved.
 \end{theorem}

\subsection{Consequences for the type fan of alcoved polytopes}\label{subsec:type_fan}

The cone of parameters $a_{i,j}$ from~\eqref{eq:alcoved} minimally defining an alcoved polytope is naturally subdivided into regions parametrized by the normal fans of alcoved polytopes.
Given two alcoved polytopes $P$ and $Q$ with the same normal fan defined by parameters $a_{i,j}$ and $b_{i,j}$, their Minkowski sum $P+Q$ has the same normal fan and is defined by the parameters $a_{i,j}+b_{i,j}$.
So each region of alcoved polytopes with fixed normal fan forms a cone in $\R^{(n-1)n}$ which is open in the relative topology. In particular, the subdivision by normal fans defines a fan structure on the cone of alcoved polytopes. We call this the \emph{type fan of alcoved polytopes} and denote it by $\cF_n$.

This fan structure also encodes the compatibility of alcoved polytopes as posed in Problem~\ref{prob:sum_alcoved_polytopes}.

\begin{proposition}\label{prop:compatibility_type_fan}
	Let $P,Q\subset \cH_n$ be two alcoved polytopes.
	Suppose $P$ and $Q$ are minimally defined by the parameters  $\mathbf{a}=(a_{i,j})_{1\le i,j\le n}$ and  $\mathbf{b}=(b_{i,j})_{1\le i,j\le n}$ via~\eqref{eq:alcoved} in the fan $\cF_n\subseteq \R^{(n-1)n}$, respectively.
	Then $P+Q$ is an alcoved polytope if and only if there exists a cone $C\in \cF_n$ such that both $\mathbf{a}$ and $\mathbf{b}$ are in $C$.
\end{proposition}
\begin{proof}
	Let $P,Q\subseteq \cH_n$ be two alcoved polytopes minimally defined by the points $\mathbf{a}, \mathbf{b}\in \R^{(n-1)n}$, respectively.
	We begin with two observations which immediately follow from the definition of the normal fan of a polytope.
	\begin{enumerate}
		\item The normal fan of the Minkowski sum $P+Q$ is the common refinement of the normal fans $\Sigma_P$ and $\Sigma_Q$.
		\item\label{property:2} Suppose that $\mathbf{a}$ and $\mathbf{b}$ are in the interior of the cones $C_a$ and $C_b$ in $\cF_n$, respectively.
		The normal fan $\Sigma_P$ is a refinement of the normal fan $\Sigma_Q$ if and only if the cone $C_b$ is a face of the cone $C_a$ in $\cF_n$.
	\end{enumerate}
	With this prelude we can conclude that if the polytopes $P$ and $Q$ are compatible then the common refinement $\Sigma:=\Sigma_P\cap \Sigma_Q$ of their normal fans is alcoved, i.e., the normal fan of some alcoved polytope.
	The normal fan $\Sigma$ thus corresponds to a point $\mathbf{s}\in \cF_n$ lying in the interior of some cone $C_s$ of $\cF_n$.
	As $\Sigma$ is by construction a refinement of both $\Sigma_P$ and $\Sigma_Q$,  both cones $C_a$ and $C_b$ are faces of the cone $C_s$ (by property~\eqref{property:2} above).
	Hence, both $\mathbf{a}$ and $\mathbf{b}$ are contained in the cone $C_s$.
	Reversing these arguments yields the converse which completes the proof.
\end{proof}

Theorem~\hyperref[thm:A]{A} has the following consequence for the type fan of alcoved polytopes.

\begin{theorem}\label{thm:2_determined}
    The type fan of alcoved polytopes is two-determined, i.e., if in a collection of rays $\rho_1,\ldots , \rho_s$ every pair $\rho_i,\rho_j$ belongs to some cone of $\cF_n$, then there exists a cone containing the whole collection.
\end{theorem}

Theorem~\ref{thm:2_determined} implies that to describe the combinatorics of the type fan $\cF_n$ it is enough to know its rays, and its two-dimensional cones.
The rays of the type fan $\cF_n$ are parametrized by (normal fans of) indecomposable alcoved polytopes, i.e., alcoved polytopes which do not admit nontrivial Minkowski sum decompositions with alcoved polytope summands. 

One family of indecomposable alcoved polytopes is formed by alcoved simplices which are classified by ordered set partitions.
We investigate alcoved simplices in detail in~\Cref{sec:alcoved_simplices}.
Another class of alcoved polytopes which are indecomposable is formed by  connected positroids in the sense of matroid theory. Indeed, positroid polytopes are alcoved polytopes such that their edges are parallel to roots (after a coordinate transformation described below). Hence, if a positroid polytope is decomposed as a Minkowski sum of two alcoved polytopes, both summands have to be positroid polytopes as well. In particular, positroid polytopes are indecomposable into alcoved summands if and only if the positroid is connected.
Combinatorially, connected positroids are parametrized by stabilized-interval-free permutations which are permutations that do not stabilize any proper interval~\cite{ARW16}. In this work we mostly focus on compatibility of alcoved simplices, however we believe that our techniques could be applied to study compatibility of positroid polytopes as well.

\subsection{Connection to the all-subset hyperplane arrangement}\label{sec:all_subset} 
The \emph{all-subset hyperplane arrangement} in $\cH_n$
consists of all hyperplanes $\sum_{j \in S} x_j=0$ for all proper, nonempty subsets $S$ of $\lbrack n\rbrack$. 
The all-subset arrangement, sometimes also called the \emph{resonance arrangement}, has been extensively studied in the recent years \cite{early2017canonical,early2018honeycomb,gutekunst2019root,kuehne23,liu2023adjoint}.
In particular, maximal chambers of the all-subset arrangement are in bijection with certain generalized retarded Green's functions which occur in thermal field theory in theoretical particle physics \cite{evans1995being} and are enumerated in O.E.I.S. entry A034997. 
In this subsection we explain the relation between the all-subset arrangement and alcoved polytopes.  
\begin{proposition}\label{prop:refinement}
	The common refinement of the normal fans of all alcoved polytopes is equal to the chamber fan for the all-subset hyperplane arrangement in its coarsest fan structure.
\end{proposition}
\begin{proof}
	The common refinement of the normal fans of all alcoved polytopes refines the chamber fan for the all-subset hyperplane arrangement, since the hyperplanes in the all-subset hyperplane arrangement are exactly the normal fans of the one-dimensional alcoved polytopes, i.e., line segments.  Conversely, for any alcoved polytope,  any cone in its normal fan is bounded by hyperplanes in the all-subset arrangement, hence the cones themselves are unions of cones in the chamber fan for the all-subset hyperplane arrangement.
\end{proof}

For a fan $\Sigma$, a convex polytope $P$ such that the normal fan $\Sigma_P$ coarsens $\Sigma$ is uniquely determined by its support function. Conversely, every convex cone-wise linear function on $\Sigma$ is the support function of some polytope with normal fan coarser than $\Sigma$. Thus the set of polytopes with normal fan coarser than $\Sigma$ naturally form a polyhedral cone. Faces of the above cone correspond to the possible normal fans of polytopes with normal fan coarser than $\Sigma$.
Let us denote by $R_n$ the polytopes whose normal fans are coarser than the all-subset arrangement.
Since the all-subset arrangement is a refinement of any normal fan of alcoved polytope, their normal fans correspond to (some) faces of $R_n$. Therefore we obtain the following corollary of \Cref{prop:refinement}.
\begin{corollary}
    The type fan of alcoved polytopes $\cF_n$ is combinatorially equivalent to a union of faces of $R_n$.
\end{corollary}

\section{Intersection of root cones}\label{sec:root_cones}

In this section we set up a general framework to study the intersection of cones generated by roots.
We will use these techniques in the following sections to discuss the common refinement of the normal fans of alcoved simplices.

A root cone $\sigma$ is a (in general not pointed) cone in $\R^n$ generated by roots.
We define a partially directed graph $\Gamma_\sigma$ which captures the combinatorics of the roots in $\sigma$.
A partially directed graph is a graph that contains both directed and undirected edges.
We think of this as a simplification of directed graphs where we replace a pair of directed edges $i\to j$ and $j\to i$ by an undirected edge $\{i,j\}$.
In what follows we treat directed graphs as partially directed graphs using this identification.

\begin{definition}
	For a root cone $\sigma\subseteq \R^n$ we define a partially directed graph $\Gamma_\sigma$ in the following way.
	\begin{enumerate}
		\item The vertices of $\Gamma_\sigma$ are labeled by $[n]$;
		\item $\Gamma_\sigma$ has an undirected edge $\{i,j\}$ if both roots $e_{ij}$ and $e_{ji}$ are in $\sigma$; 
		\item $\Gamma_\sigma$ has a directed edge $i\to j$ for every root $e_{ij}\in \sigma$ (and $e_{ji}\not \in \sigma$).
	\end{enumerate}
\end{definition}

\begin{remark}
	Note that if $\sigma$ is a pointed root cone, i.e., $\sigma$ does not contain any vector space, the graph $\Gamma_\sigma$ does not have undirected edges. 
	
	On the other end of the spectrum, if $\sigma=L$ is a linear space, the graph $\Gamma_\sigma$ does not have directed edges and coincides with the graph $\Gamma_L$ from Definition~\ref{def:GL}.
	
	Now, let  $\sigma$ be a general root cone with the lineality space $L_\sigma$ and let $\sigma'$ be the  
	corresponding pointed root cone in the quotient $\cH_n/L_\sigma$. Then, $\Gamma_{L_\sigma}$ is the union of all the undirected edges of $\Gamma_\sigma$ and $\Gamma_{\sigma'}$ is obtained form $\Gamma_\sigma$ by contracting all undirected edges and then removing parallel directed edges.
\end{remark}

By a directed path from $i$ to $j$ in a partially directed graph we mean a sequence of edges starting in $i$ and ending in $j$ where undirected edges can be crossed in both directions and directed edges only in their specified directions.

A partially directed graph $G$ is \emph{transitively closed} if there exists an edge $i\to j$ in $G$ whenever $G$ has a directed path from $i$ to $j$.
For a given graph $G$, we define its transitive closure denoted by $\tc (G)$ as the smallest partially directed transitively closed graph containing $G$.

\begin{lemma}\label{lem:transitively_closed}
	The partially directed graph $\Gamma_\sigma$ of a root cone $\sigma\subseteq\R^n$ is transitively closed. Moreover, let $\sigma_A$ be a cone generated by a set of roots $A$. Then the graph $\Gamma_{\sigma_A}$ is the transitive closure of the graph $\Gamma_A$ having a directed edge $i\to j$ for every root $e_{ij}\in A$.
\end{lemma}
\begin{proof}
	For the first part, assume there is a path from $i=i_1\to i_2 \to \dots \to i_k = j$ in $\Gamma_\sigma$ then $\sigma$ contains the roots $e_{i_si_{s+1}}$ for $1\le s\le k-1$.
	Therefore, $e_{i_1i_2}+\dots+e_{i_{k-1}i_k}=e_{ij}\in \sigma$.
	
	For the second part notice that there is a directed path from $i$ to $j$ in $\Gamma_A$ if and only if the root $e_{ij}$ can be expressed as a non-negative combination of roots in $A$. Thus $\Gamma_{\sigma_A}=\tc (\Gamma_A)$ by the definition of both graphs.
\end{proof}

Analogously to the discussion in~\Cref{sec:type_fan} we can describe the set of points in $\sigma$ as a collection of weights on edges of $\Gamma_\sigma$, which are non-negative for directed edges. Similarly, a collection of weights defines the same point $x\in \sigma$ if and only if their difference corresponds to a linear combination of directed cycles in $\Gamma_\sigma$.
To study the intersection of two root cones $\sigma$ and $\tau$ we define the \emph{intersection graph of} $\sigma$ and $\tau$ as a partially directed graph $\Gamma_{\sigma,\tau}$ that encodes the rays of $\sigma\cap \tau$.
\begin{definition}
	Concretely, we define the partially directed graph $\Gamma_{\sigma,\tau}$ as $\Gamma_\sigma\cup \Gamma^{\op}_{\tau}$, where $\Gamma^{\op}_{\tau}$ is the opposite graph of $\Gamma_\tau$, i.e., the same graph with all directed edges reversed.
	To record for each edge in~$\Gamma_{\sigma,\tau}$ whether it stems from $\Gamma_\sigma$ or $\Gamma^{\op}_\tau$ we call the edges of $\Gamma_\sigma$ \emph{upper} and of $\Gamma^{\op}_\tau$ \emph{lower} edges, respectively, where we have a drawing of the graph $\Gamma_{\sigma,\tau}$ as in Figure~\ref{fig:Gsigma} in mind.
	Note that if an edge appears both in $\Gamma_\sigma$ and $\Gamma^{\op}_\tau$ we keep it twice in $\Gamma_{\sigma,\tau}$ once as an upper and once as a lower edge (as for instance the edge $1\to 2$ in Figure~\ref{fig:Gsigma}).
\end{definition}

\begin{definition}\label{def:primalt}
	Let $C$ be a cycle in an intersection graph $\Gamma_{\sigma,\tau}$.
	We call the cycle $C$ \emph{alternating} if the upper and lower edges in $C$ alternate, i.e., no two consecutive edges in $C$ are both upper or both lower.
	
	We call an alternating cycle $C$ \emph{primitive} if there is no collection of non-empty alternating cycles $C_1,\ldots, C_k$ with $k\geq 2$ such that the sets of heads and tails of upper edges of $C$ coincides with the union of the sets of heads and tails of upper edges of $C_1,\ldots,C_k$.
\end{definition}

With these definitions we can give a graphical description of the intersection of two root cones.

\begin{proposition}\label{prop:root_cones}
	Let $\sigma, \tau$ be two root cones. 
	\begin{enumerate}[(i)]
		\item The lattice points in the intersection $\sigma\cap\tau$ are given by collections of cycles in $\Gamma_{\sigma,\tau}$. 
		More precisely if $\mathcal{C}$ is a set of cycles in $\Gamma_{\sigma,\tau}$ then
		\[
		x_\mathcal{C} = \sum_{(i\to j) \in \mathcal{C}\cap \Gamma_{\sigma}} e_{ij}
		\]
		is a lattice point in $\sigma\cap\tau$ where this sum takes the number of appearances of $e_{ij}$ in the collection of cycles $\mathcal{C}$ into account.
		Moreover, every lattice point arises in this way.
		However, several collections of cycles can define the same point in $\sigma\cap\tau$.
		If $\mathcal{C}$ contains just a single cycle $C$ we also denote $x_\mathcal{C}$ by $x_C$.
		\item For every cycle $C$ in $\Gamma_{\sigma,\tau}$ there exists an alternating cycle $C'$ with $x_C=x_{C'}$.
		Therefore it suffices to consider the alternating cycles of $\Gamma_{\sigma,\tau}$.
		\item Assume that $\sigma\cap \tau$ is a pointed cone.
		In this case, every ray of $\sigma\cap \tau$ is generated by a point $x_C$ where $C$ is a primitive alternating cycle of~$\Gamma_{\sigma,\tau}$.
		And conversely, for every primitive alternating cycle $C$ there is a ray generated by $x_C$.
	\end{enumerate}
\end{proposition}

Before proving this proposition we give an example.
\begin{example}\label{ex:root_cones}
	Consider the two roots cones $\sigma,\tau \in \R^4$ defined by
	\[
	\sigma = \langle e_{12},e_{23},e_{34} \rangle,\quad \tau = \langle e_{14},e_{21},e_{32} \rangle.
	\]
	Their corresponding graphs $\Gamma_\sigma,\Gamma_\tau,\Gamma_{\sigma,\tau}$ are depicted in Figure~\ref{fig:Gsigma}.
	The intersection graph $\Gamma_{\sigma,\tau}$ has the four primitive alternating cycles
	\[
	C_1: \,1\to 4 \to 1, \quad C_2: \,2 \to 4 \to 2, \quad C_3: \, 3 \to 4 \to 3, \mbox{ and } C_4: \, 1 \to 2 \to 3 \to 4 \to 1.
	\]
	The cycle $C_4$ is the one highlighted in Figure~\ref{fig:Gsigma} in bold.
	The corresponding rays are generated by
	\[
	x_{C_1}=e_{14}, \quad x_{C_2}=e_{24}, \quad x_{C_3}= e_{34}, \mbox{ and } x_{C_4}= e_{12}+e_{34}.
	\]
	Note that the ray generated by $x_{C_4}$ is not generated by a root.
\end{example}
	\begin{center}
		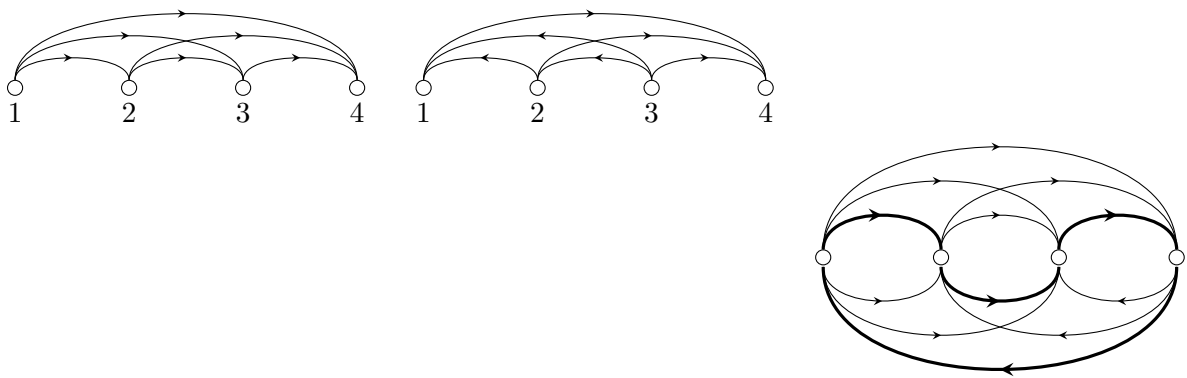
\begin{figure}[hbt]
			\begin{subfigure}{.3\linewidth}
				\begin{center}
				\begin{tikzpicture}[baseline=0,>=stealth,xscale=1.5]
        			\foreach \x in {1,...,4}{
        				\coordinate (\x) at (\x,0);
        			}
        			\foreach \x in {1,2,3,4}{
        				\draw (\x) node[below,fill=white,draw=black,circle,inner sep =2] {} node[below,yshift=-5pt] {$\x$};
        			}
        			\begin{scope}[decoration={
        					markings,
        					mark=at position 0.5 with {\arrow{>}}}
        				]
        				\draw[postaction={decorate}] (1) to [in=90,out=90] (2);
        				\draw[postaction={decorate}] (2) to [in=90,out=90] (3);
        				\draw[postaction={decorate}] (3) to [in=90,out=90] (4);
        				\draw[postaction={decorate}] (1) to [in=90,out=90] (3);
        				\draw[postaction={decorate}] (2) to [in=90,out=90] (4);
        				\draw[postaction={decorate}] (1) to [in=90,out=90] (4);
        			\end{scope}
        		\end{tikzpicture}
				\end{center}
			\end{subfigure}
            \hspace{.02\linewidth}
			\begin{subfigure}{.3\linewidth}
				\begin{center}
					\begin{tikzpicture}[baseline=0,>=stealth,xscale=1.5]
            			\foreach \x in {1,...,4}{
            				\coordinate (\x) at (\x,0);
            			}
            			\foreach \x in {1,2,3,4}{
            				\draw (\x) node[below,fill=white,draw=black,circle,inner sep =2] {} node[below,yshift=-5pt] {$\x$};
            			}
            			\begin{scope}[decoration={
            					markings,
            					mark=at position 0.5 with {\arrow{>}}}
            				]
            				\draw[postaction={decorate}] (1) to [in=90,out=90] (4);
            				\draw[postaction={decorate}] (2) to [in=90,out=90] (1);
            				\draw[postaction={decorate}] (3) to [in=90,out=90] (2);
            				\draw[postaction={decorate}] (2) to [in=90,out=90] (4);
            				\draw[postaction={decorate}] (3) to [in=90,out=90] (1);
            				\draw[postaction={decorate}] (3) to [in=90,out=90] (4);
            			\end{scope}
            		\end{tikzpicture}
				\end{center}
			\end{subfigure}
            \hspace{.02\linewidth}
			\begin{subfigure}{.3\linewidth}
				\begin{center}
						\begin{tikzpicture}[baseline=0,>=stealth,scale=1.55]
        			\foreach \x in {1,...,4}{
        				\coordinate (\x) at (\x,-1.45);
        			}
         			\coordinate (1a) at (1,-1.6);
        			\coordinate (2a) at (2,-1.6);
        			\coordinate (3a) at (3,-1.6);
        			\coordinate (4a) at (4,-1.6);
        			\foreach \x in {1,2,3,4}{
        				\draw (\x) node[below,fill=white,draw=black,circle,inner sep =2] {};
        			}
        			\begin{scope}[decoration={
        					markings,
        					mark=at position 0.5 with {\arrow{>}}}
        				]
        				\draw[very thick,postaction={decorate}] (1) to [in=90,out=90] (2);
        				\draw[postaction={decorate}] (2) to [in=90,out=90] (3);
        				\draw[very thick,postaction={decorate}] (3) to [in=90,out=90] (4);
        				\draw[postaction={decorate}] (1) to [in=90,out=90] (3);
        				\draw[postaction={decorate}] (2) to [in=90,out=90] (4);
        				\draw[postaction={decorate}] (1) to [in=90,out=90] (4);
        				
        				\draw[very thick,postaction={decorate}] (4a) to [in=-90,out=-90] (1a);
        				\draw[postaction={decorate}] (1a) to [in=-90,out=-90] (2a);
        				\draw[very thick,postaction={decorate}] (2a) to [in=-90,out=-90] (3a);
        				\draw[postaction={decorate}] (4a) to [in=-90,out=-90] (2a);
        				\draw[postaction={decorate}] (1a) to [in=-90,out=-90] (3a);
        				\draw[postaction={decorate}] (4a) to [in=-90,out=-90] (3a);
        			\end{scope}
        		\end{tikzpicture}
			\end{center}
			\end{subfigure}
			\caption{The graphs $\Gamma_\sigma$, $\Gamma_\tau$, and $\Gamma_{\sigma,\tau}$ for the root cones described in Example~\ref{ex:root_cones}. A primitive alternating $4$-cycle of $\Gamma_{\sigma,\tau}$ is highlighted in bold.}
            \label{fig:Gsigma}
		\end{figure}
	\end{center}

\begin{proof}[Proof of Proposition~\ref{prop:root_cones}]
	For (i) let $C$ by a cycle in $\Gamma_{\sigma,\tau}$.
	Denote by $C^\uparrow$ and $C^\downarrow$ the upper and lower edges in $C$, respectively.
	Since $C$ is a cycle in $\Gamma_{\sigma,\tau}$ we have
	\[
	\sum_{(i\to j) \in C^\uparrow} e_{ij}+\sum_{(i'\to j') \in C^\downarrow} e_{i'j'}=0
	\]
	and thus
	\[
	\sum_{(i\to j) \in C^\uparrow} e_{ij}=\sum_{(i'\to j') \in C^\downarrow} e_{j'i'}.
	\]
	Note that by definition both sides of the equation equal $x_C$ and the left-hand side is by construction a lattice point in $\sigma$ and the right-hand side a lattice point in $\tau$.
	
	To prove (ii) note that the graphs $\Gamma_\sigma$ and $\Gamma^{\op}_\tau$ are transitively closed by Lemma~\ref{lem:transitively_closed}.
	Thus given a cycle $C$ in $\Gamma_{\sigma,\tau}$ we can replace every consecutive sequence of only upper or only lower edges with one single edge to obtain an alternating cycle $C'$ that satisfies $x_C=x_{C'}.$ 
	
	To prove (iii) notice that a vector $x\in \sigma$ is a ray generator of a cone $\sigma$ if it cannot be written as a linear combination $x=z+y$ of two non-collinear vectors $z,y\in \sigma$. Thus part (iii) follows directly from part (i), (ii) and the definition of primitive cycles.
\end{proof}

Proposition~\ref{prop:root_cones} yields a graphic criterion to decide whether the intersection of two root cones is again a root cone.
Note that a closely related criterion appeared in Postnikov's work in his study of triangulations of root polytopes~\cite{Postnikov}.

\begin{corollary}\label{cor:pointed_cones_comp}
	Let $\sigma,\tau$ be two root cones and assume that their intersection $\sigma\cap \tau$ is a pointed cone. Then $\sigma\cap \tau$ is a root cone if and only if the graph $\Gamma_{\sigma,\tau}$ does not have a primitive alternating cycle of length at least four.
\end{corollary}
\begin{proof}
	Indeed, by  Proposition~\ref{prop:root_cones} $(iii)$ rays of $\sigma\cap \tau$ are in bijection with primitive alternating cycles in $\Gamma_{\sigma,\tau}$.
	The vectors corresponding to these cycles are roots if and only if the cycle is of length at most two. Since the length of an alternating cycle is even this finishes the proof.
\end{proof}

\begin{remark}
    Postnikov, Reiner and Williams studied the polar duals of root cones, so-called \emph{braid cones} in~\cite[Sec. 3]{PRW08}.
    They describe a dictionary betwen braid cones and preposets, the latter being equivalent to the transitively closed partially directed graphs that we study in this article.
    While both articles study essentially the same cones, the results are mostly disjoint as the focus of the relevant parts of~\cite{PRW08} lies on describing braid cones as union of Weyl cones and their relation to generalized permutohedra.
\end{remark}

\subsection{Non pointed case}

In this subsection we describe the graphical compatibility criterion of root cones $\sigma,\tau$ generalizing Corollary~\ref{cor:pointed_cones_comp} to the case when $\sigma \cap \tau$ is not necessarily pointed.
By construction, the compatibility can be checked in two steps:
\begin{enumerate}
	\item First, we check that the intersection of the lineality spaces $L_\sigma\cap L_\tau$ of $\sigma, \tau$ is a root subspace. To do so, apply Lemma~\ref{lem:cycles} to the undirected part of $\Gamma_{\sigma,\tau}$ and check that every cycle in this graph of length at least four contains a chord.
	\item If $L_\sigma\cap L_\tau$ is a root space, it is left to work with $\sigma\cap \tau/(L_\sigma\cap L_\tau)$ which is a pointed cone in the quotient space so we can apply Corollary~\ref{cor:pointed_cones_comp} to it. The only step left is to describe $\Gamma_{\sigma/(L_\sigma\cap L_\tau)}, \Gamma_{\tau/(L_\sigma\cap L_\tau)}$. These graphs are obtained from $\Gamma_\sigma$ and $\Gamma_\tau$ by contracting the undirected edges corresponding to the roots in $(L_\sigma\cap L_\tau)$.
\end{enumerate}

The next theorem shows that these two steps can be checked analogously to Corollary~\ref{cor:pointed_cones_comp} above.

\begin{theorem}\label{thm:cone_criterion}
	Let $\sigma, \tau$ be two root cones (possibly linear subspaces), then the intersection $\sigma\cap \tau$ is a root cone if and only if the graph $\Gamma_{\sigma,\tau}$ does not have a primitive alternating cycle of length greater or equal to four.
\end{theorem}
\begin{proof}
	Let $K_1,\dots,K_r$ denote the components of the intersection of the undirected parts of $\Gamma_\sigma$ and $\Gamma_\tau$.
	These are cliques as both graphs are transitively closed.
	
	Step 1 of the strategy above looks for alternating cycles in the undirected part of the graph $\Gamma_{\sigma,\tau}$, which are primitive alternating cycles in $\Gamma_{\sigma,\tau}$ in the sense of Definition~\ref{def:primalt}. We need to check whether every such cycle of length at least four has a chord.
	
	Step 2 looks for primitive alternating cycles in the quotient of the graph $\Gamma_{\sigma,\tau}$ by the common undirected cliques $K_1\dots,K_r$. The existence of the primitive alternating cycle of length $\ell$ in this quotient is equivalent to the existence of the primitive alternating cycle of length $\ell$ in the original graph $\Gamma_{\sigma,\tau}$.
	This is indeed the case as a primitive cycle in $\Gamma_{\sigma,\tau}$ can involve at most one vertex of each undirected clique $K_1\dots,K_r$ as the cycle would not be primitive otherwise.
	
	For the other direction, let $C$ be a primitive alternating cycle in the quotient graph. Assume it is passing through vertices corresponding to cliques $K_{i_1},\dots,K_{i_s}$. Then a choice of a vertex $v_j\in K_{i_j}$ defines a corresponding primitive alternating cycle $\widetilde C$ in $\Gamma_{\sigma,\tau}$.

	Hence both tests above provide a positive result if and only if $\Gamma_{\sigma,\tau}$ does not have a primitive alternating cycle of length greater or equal to $4$.
\end{proof}

\begin{remark}
Theorem~\ref{thm:cone_criterion} give a potential criterion for compatibility of any pair of alcoved polytopes $P$ and $Q$ assuming we know their normal fans $\Sigma_P$ and $\Sigma_Q$. Indeed, $P+Q$ is alcoved if and only if $\tau\cap\sigma$ is a root cone for any pair $\tau\in \Sigma_P, \sigma\in \Sigma_Q$. In the next sections we carry out this strategy for the pairs of alcoved simplices, for which we have a convenient description for the cones of the normal fan. 

It is an interesting question  how to get a convenient description of the normal fan for more general alcoved polytopes. A step in this direction is given by \cite[Theorem 11]{joswig2016weighted} describing face lattice of braid cones (that is alcoved polyhedra which are cones).
\end{remark}

\section{Alcoved simplices}\label{sec:alcoved_simplices}

An \emph{ordered set partition}  of the set $\left[n\right]$ is an ordered tuple $\S = (B_1,\ldots, B_\ell)$ of pairwise disjoint subsets $B_i\subseteq \left[n\right]$ with $\cup_{j=1}^\ell B_j = \left[n\right]$.
We denote the set of ordered set partitions of $\left[n\right]$ by~$\osp(n)$.

Moreover, we use the shorthand notation $(1,23,4)$ for the ordered set partition $(\{1\},\{2,3\},\{4\})$ in $\osp(4)$.
An ordered set partition $\S = (B_1,\ldots, B_\ell)$ is called \emph{nondegenerate} if each block $B_i$ is a singleton, i.e.,\ contains exactly one element of $\lbrack n \rbrack$.

\begin{definition}\label{def:alcoved_simplex}
	To each ordered set partition $\S=(B_1,\dots,B_\ell)$ of $\lbrack n \rbrack$ we associate an \emph{alcoved simplex} $\Delta_\S$ in the hyperplane $\cH_n$ defined by the following set of (in)equalities in $\cH_n$:
	
	\begin{equation}\label{eq:simplex}
		\begin{aligned}
			x_{i}= x_{j}\quad \text{ for every }& i,j\in B_k \text{ and every }1\le k \le \ell,\\
			x_i\geq x_j \quad \text{ for every }& i\in B_k, \, j\in B_{k+1} \text{ and every }1\le k \le \ell-1,\\
			x_i\geq x_j-1 \quad \text{ for every }& i\in B_\ell, \, j\in B_1.
		\end{aligned}
	\end{equation}
	We denote by $\Sigma_{\S}$ the normal fan of the simplex $\Delta_\S$.
\end{definition}

\begin{definition}
	Let $\S=(B_1,\ldots,B_{\ell})$ be an ordered set partition of $[n]$. We define a graph $G_\S$ as a partially directed graph on $n$ vertices which has an undirected clique on the set $B_i$  for every $1\leq i \leq \ell$ and a directed edge $b_i\to b_{i+1}$ for  $1\leq i \leq \ell$ (regarded cyclically) where $b_j\in B_j$ is the smallest element of a block $B_j$.
\end{definition}

\begin{example}\label{ex:osp}
	The alcoved simplex $\Delta_{(1,23,4)}$ in $\R^4$ of the ordered set partition $(1,23,4)$ is defined by $x_1+\dots+x_4=0$ and the (in)equalities
	\[
	x_1\ge x_2=x_3\ge x_4\ge x_1-1.
	\]
	Its vertices are $(0,0,0,0)$, $\left(\frac{3}{4},-\frac{1}{4},-\frac{1}{4},-\frac{1}{4}\right)$ and $\left(\frac{1}{4},\frac{1}{4},\frac{1}{4},-\frac{3}{4}\right)$.
	The graph $G_{(1,23, 4)}$ is depicted in Figure~\ref{fig:graph_gamma_S}.
\end{example}

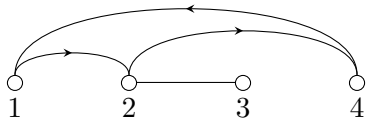
\begin{figure}[hbt]
	\begin{center}
		\begin{tikzpicture}[baseline=0,>=stealth,xscale=1.5]
			\foreach \x in {1,...,4}{
				\coordinate (\x) at (\x,0);
			}
			
			\draw  (2,-0.1) to (3,-0.1);
			\foreach \x in {1,2,3,4}{
				\draw (\x) node[below,fill=white,draw=black,circle,inner sep =2] {} node[below,yshift=-5pt] {$\x$};
			}
			\begin{scope}[decoration={
					markings,
					mark=at position 0.5 with {\arrow{>}}}
				]
				\draw[postaction={decorate}] (1) to [in=90,out=90] (2);

				\draw[postaction={decorate}] (4) to [in=90,out=90] (1);

				\draw[postaction={decorate}] (2) to [in=90,out=90] (4);
				
			\end{scope}
		\end{tikzpicture}
		\caption{The graph $G_\S$ of the ordered set partition $\S=(1,23,4)$.}
		\label{fig:graph_gamma_S}
	\end{center}
\end{figure}

As the next step we give a description of the normal fan $\Sigma_{\S}$ of $\Delta_\S$; this was taken as the definition of permutohedral blades, denoted $((B_1,\ldots, B_\ell))$ in \cite{early2018honeycomb}.

\begin{proposition}
	Let $\S=(B_1,\dots,B_\ell)$ be an ordered set partition.
	\begin{enumerate}[(i)]
	\item The normal fan $\Sigma_\S$ of $\Delta_\S$ has a lineality space $L_\S$ in $\cH_n$ spanned by the roots corresponding to undirected edges in $G_\S$:
	\[
	L_\S= \langle e_{ij}\mid \{i,j\}\in G_\S\rangle.
	\]
	\item The generators of the rays of $\Sigma_\S/L_\S$ are the directed edges of $G_\S$.
	The cones of $\Sigma_\S/L_\S$ are in bijection with the proper subsets of these rays.
	\item The polytope $\Delta_\S$ is an alcoved simplex of dimension $\ell-1$.
	\end{enumerate}
\end{proposition}

\begin{proof}
	The claim $(i)$ follows directly from the construction of $\Delta_\S$ as an equation $x_i=x_j$ corresponds  to the undirected edge $\{i,j\}\in G_\S$ which yields the generator $e_{ij}$ in the lineality space.
	
	For $(ii)$ let $b_i$ be the smallest element in the block $B_i$ for every $1\le i \le \ell$. 
	Using the equations $x_i=x_j$~\eqref{eq:simplex} we can rewrite and reduce the inequalities in the definition of $\Delta_\S$ to
	\[
	x_{b_1}-x_{b_2}\ge 0, \quad \dots\quad  	x_{b_\ell-1}-x_{b_\ell}\ge 0,	\quad x_{b_\ell}-x_{b_1}\ge -1.
	\]
	The corresponding rays in $\Sigma_\S/L_\S$ thus exactly correspond to the directed edges of $G_\S$ and are generated by $e_{b_ib_{i+1}}$ for $1\le i \le \ell$ (regarded cyclically).
	
	As these are $\ell$ rays in $\cH_n/L_n$ which is of dimension $\ell-1$, $\Sigma_\S/L_\S$ is the normal fan of a simplex of dimension $\ell-1$.
	Thus, every cone of this fan arises by omitting at least one of these rays and every such proper subset of the rays generates a cone of this fan.
	As both the lineality space and the rays in the quotient fan are generated by roots, the polytope $\Delta_\S$ is alcoved which also proves $(iii)$.
\end{proof}

Note that by the above description, the normal fan $\Sigma_\S$ is invariant under cyclic shifts of $\S$.
\begin{corollary}
	Let $\S=(B_1,\dots,B_\ell)$ be an ordered set partition and let $\S'=(B_2,\dots,B_\ell,B_1)$ be the ordered set partition where the blocks are cyclically shifted to the left by one position.
	Then we have
	\[
	\Sigma_{\S}=\Sigma_{\S'}.
	\]
\end{corollary}
As we are mostly interested in the common subdivisions of the normal fans of alcoved simplices, we consider the ordered set partitions up to cyclic shifts.
Thus, we assume that for an ordered set partition $\S=(B_1,\dots,B_\ell)$ of $\lbrack n \rbrack$ we have $n\in B_\ell$.

\begin{example}
	We continue the discussion of the ordered set partition $\S=(1,23,4)$ introduced in Example~\ref{ex:osp}.
	It has the directed edges $\{1 \to 2,2 \to 4, 4\to 1\}$ and the undirected edge $\{2,3\}$.
	Thus the lineality space $L_\S$ is $1$-dimensional and spanned by $e_{23}$.
	The quotient fan $\Sigma_{\S}/L_\S$ thus has three maximal cones positively spanned by $\{e_{12},e_{24}\}$, $\{e_{12},e_{41}\}$ and $\{e_{24},e_{41}\}$.
	All other cones in this fan are the faces of these maximal cones.
\end{example}

As we discussed above, every set partition $\S$ yields an alcoved polytope $\Delta_\S$.
The next proposition states that up to shifts and dilations, all alcoved simplices arise in this way.
\begin{proposition}\label{prop:every_simplex}
	Every alcoved simplex in $\cH_n$ is equal to $\Delta_\S$ for some ordered set partition $\S$ of $[n]$ up to shift and dilation.
\end{proposition}
\begin{proof}
	A simplex is determined up to translation and dilation by its normal fan.
	Hence it is enough to show that the normal fan of an alcoved simplices is the normal fan of the simplex $\Delta_\S$ for some set partitions $\S$.
	
	The normal fan of an $n-1$-dimensional simplex $\Delta$ has $n$ rays with exactly one linear dependency $\sum_{i=1}^n \lambda_i v_i=0$ between the ray generators $v_1,\ldots,v_n$.
	Moreover, the coefficients $\lambda_i$ in the above linear relation can be chosen to be strictly positive for all $i$.
	In other words, this means that $0$ belongs to the convex hull of $v_1,\dots, v_n$.
	
	Now, let $\Delta$ be a full-dimensional alcoved simplex in $\cH_n$. Then we can choose the generators of the rays $v_1,\dots,v_n$ of its normal fan to be roots.
	Therefore the linear dependency $\sum_{i=1}^n \lambda_i v_i=0$ defines a collection of cycles in $G_{\cH_n}$.
	However, since the above relation is unique and involves all roots $v_1,\dots,v_n$, it defines a unique oriented cycle of size $n$.
	This defines a cyclic order on $[n]$, i.e., a nondegenerate ordered set partition $\S$. By construction of $\Delta_\S$, we get that $\Sigma_\Delta= \Sigma_\S$, which finishes the proof for full-dimensional alcoved simplices.
	
	If $\Delta$ is an alcoved simplex in $\cH_n$ of strictly lower dimension we consider $\Delta$ in its affine span $\text{Aff}_\Delta$.
    By translating $\Delta$, we may assume that $\text{Aff}_\Delta$ is a linear subspace.
	As $\Delta$ is alcoved, the subspace $\text{Aff}_\Delta$ is defined by $x_{i_1}-x_{j_1}=\dots=x_{i_k}-x_{j_k}=0$ for a set of roots $e_{i_1j_1},\dots,e_{i_kj_k}$ with $i_s<j_s$ for $1\le s \le k$.
	Hence the simplex $\Delta$ is by definition a full-dimensional alcoved simplex in the subspace $\text{Aff}_\Delta$.
	So by the previous argument its normal fan in $\text{Aff}_\Delta$ agrees with the normal fan of some $\Delta_\S$ for some $\S$.
\end{proof}

The vertices of the simplex $\Delta_\S$ are not integral in general.
There is however a natural transformation that allows to represent these simplices as the Newton polytope $N_\S$ of a polynomial which we describe now.

Fix an ordered set partition $\mathbf{S} = (B_1,\ldots, B_\ell)$ of $\left[n\right]$.
As the normal fan of $\Delta_\S$ is invariant under cyclic rotations, we can assume that $n\in B_\ell$.
We define an $(\ell-1)$-dimensional simplex $N_\S$ in $\mathbb{R}^{n-1}$ with vertices 
\[
	0,e_{B_1},e_{B_1B_2},\ldots, e_{B_1B_2\cdots B_{\ell-1}}.
\]
This simplex can be represented as a Newton polytope via
\begin{eqnarray}\label{eqn:newton}
N_\mathbf{S} = \mathrm{Newt}\left(1+\sum_{i=1}^{\ell-1} \prod_{j=1}^{i} \prod_{k\in B_j}y_k\right).
\end{eqnarray}

 \begin{proposition}\label{prop:newton}
	 	Let  $\S=(B_1,\dots,B_\ell)$ be an ordered set partition of $[n]$ such that $n\in B_\ell$.
	 	
	 	The simplex $N_\S$ equals $\pi_{\cH_n}^{-1}(\Delta_\S)\cap \{x_n=0\}$ where $\pi_{\cH_n}:\R^n \to \cH_n$ is the orthogonal projection to the hyperplane $\cH_n$.
\end{proposition}

For example, the standard ordered set partition $\mathbf{S} = (1,2,3,4)$ has the Newton polytope $N_\S$ of the polynomial $1+y_1+y_1y_2+y_1y_2y_3$ which has the vertices $(0,0,0)$, $(1,0,0)$, $(1,1,0)$ and $(1,1,1)$.  
This simplex arises under the above transformation from the alcoved simplex
\[
\Delta_\mathbf{S} = \text{conv}\left\{(0,0,0,0),\left(\frac{3}{4},-\frac{1}{4},-\frac{1}{4},-\frac{1}{4}\right),\left(\frac{1}{2},\frac{1}{2},-\frac{1}{2},-\frac{1}{2}\right) ,\left(\frac{1}{4},\frac{1}{4},\frac{1}{4},-\frac{3}{4}\right)\right\}.
\]

\begin{proof}
	Following the inequalities in~\eqref{eq:simplex},
	the simplex  $\pi_{\cH}^{-1}(\Delta_\S)\cap \{x_n=0\}$ is defined by the inequalities
		\begin{alignat*}{5}
			x_{i}=& x_{j}\quad &\text{ for every } &i,j\in B_k \text{ and every }1\le k \le \ell,\\
			x_i\geq& x_j \quad &\text{ for every } &i\in B_k, \, j\in B_{k+1} \text{ and every }1\le k \le \ell-1,\\
			x_i\geq& x_j-1 \quad &\text{ for every }& i\in B_\ell, \, j\in B_1,\\
			x_n=&0.&&
	\end{alignat*}
	Setting $\S=(\{b_1^1,\dots,b_{i_1}^1\},\dots,\{b_1^\ell,\dots,b_{i_\ell}^\ell\})$ this system of inequalities simplifies to 
	\[
		1\geq x_{b_1^1}=\dots = x_{b_{i_1}^1} \geq x_{b_1^2}=\dots = x_{b_{i_2}^2}\geq \dots \geq x_{b_1^{\ell-1}}=\dots = x_{b_{i_{\ell-1}}^{\ell-1}}\geq x_{b_1^{\ell}}=\dots = x_{b_{i_{\ell}}^{\ell}}=0.
	\]
	Thus,  $\pi_{\cH}^{-1}(\Delta_\S)\cap \{x_n=0\}$ is an $(\ell-1)$-dimensional simplex in $\R^{n-1}$ (after dropping the last coordinate).
	Moreover, the points $0,e_{B_1},e_{B_1B_2},\ldots, e_{B_1B_2\cdots B_{\ell-1}}$ all fulfill the above inequalities with equality in all but one case.
	Hence these are the $\ell$ vertices of $\pi_{\cH}^{-1}(\Delta_\S)\cap \{x_n=0\}$ which proves that this simplex is equal to $N_\S$.
\end{proof}

As this proof showed the simplex $N_\S$ is not alcoved on the nose as the coordinate vectors also appear as normal vectors.
They are however ``almost alcoved'' in the sense that all normal vectors are either roots or coordinate vectors.
They become actually alcoved by passing to the associated alcoved simplex $\Delta_\S$ under the above construction.

\section{Minkowski sums of alcoved simplices}\label{sec:compatibility}

\subsection{Graphical criterion}

In this section, we discuss a graphical criterion to detect when the Minkowski sum of two alcoved simplices $\Delta_\S$ and $\Delta_\T$ is again an alcoved polytope.

\begin{definition}
	We call two ordered set partitions $\S$ and $\T$ \emph{compatible} if the Minkowski sum $\Delta_\S+\Delta_\T$ is alcoved.
\end{definition}

We now turn to the graph theoretic side of the story.

\begin{definition}
	Let $\S, \T$ be two ordered set partitions on $\left[n\right]$.
	Let us define the partially ordered graph $G_{\S,\T}$ to be the union $G_{\S}\cup G_{\T}^{op}$ where in $G_{\T}^{op}$ all directed edges are reversed.
	We call edges in $G_\S$ \emph{upper} and those in $G_{\T}^{op}$ \emph{lower}.
	
	Let $C$ be a cycle in $G_{\S,\T}$.
	An \emph{upper path segment} of $C$ is a collection of  consecutive upper edges in $C$.
	We call a cycle \emph{violating} if it has at least two disjoint upper path segments and visits every vertex of $G_{\S,\T}$ at most once.
\end{definition}

These graphs allow us to prove a graph theoretic criterion of compatible partitions.
\begin{theorem}\label{thm:violating_cycles}
	The ordered set partitions $\S, \T$ on $\left[n\right]$ are compatible if and only if $G_{\S,\T}$ does not have a violating cycle.
\end{theorem}

\begin{proof}
First suppose $\S$ and $\T$ are not compatible.
So by definition, $\Delta_\S+\Delta_\T$ is not alcoved:

Suppose the intersection of the lineality spaces $L_\S\cap L_\T$ is not a root subspace.
Then by \Cref{lem:cycles} there is a cycle $C$ of length at least four without a chord in the graph $\Gamma_{L_\S}\cup \Gamma_{L_\T}$.
This graph coincides with the undirected part of the graph $G_{\S,\T}$ and thus $C$ is a violating cycle in $G_{\S,\T}$:
We can assume that $C$ visits every vertex at most once as we can otherwise decompose it into two smaller chordless cycles one of which of length at least four.
Secondly, it must have at least two distinct upper path segments as the upper and lower segments are cliques by~\Cref{lem:cliques} which would thus yield chords in $C$ if there was only one upper or lower segment.

So we now assume that the common lineality space $L_\S\cap L_\T$ is a root subspace.
In this case, taking the quotient $\Sigma_\S\cap\Sigma_\T/L_\S\cap L_\T$ corresponds to contracting all undirected edges in $G_{\S,\T}$ that appear both as upper and as lower edges.
Thus we can now assume that $\Delta_\S+\Delta_\T$ is a full-dimensional polytope.
Hence the assumption that $\Delta_\S+\Delta_\T$ is not alcoved means that there exist two normal cones $\sigma\in \Sigma_\S$ and $\tau\in \Sigma_\T$ such that there is a ray $\rho$ in $\sigma\cap\tau$ which is not a root.
By Theorem~\ref{thm:cone_criterion} we hence have a primitive alternating cycle $C_\Gamma$ in $\Gamma_{\sigma,\tau}$ of length at least $4$.
As every edge in $\Gamma_{\sigma,\tau}$ is a path of consecutive edges in $G_{\S,\T}$, the cycle $C_\Gamma$ yields a corresponding cycle $C_G$ in $G_{\S,\T}$.
More precisely, the edges in $C_\Gamma$ give rise to sequences of upper and lower edges in $G_{\S,\T}$ where all appearing edges are also present in the graph $\Gamma_{\sigma,\tau}$ and the concatenation of these sequences is a cycle $C_G$ in $G_{\S,\T}$.

We claim that $C_G$ is violating.
Suppose there is a vertex $v$ in $G_{\S,\T}$ that is visited at least twice in the cycle $C_G$.
By splitting up the cycle $C_G$ at the vertex $v$ we can decompose the cycle into the two cycles $C_1$ and $C_2$ in $G_{\S,\T}$ whose concatenation is $C_G$.
These cycles with the same edges form also two cycles in $\Gamma_{\sigma,\tau}$ which in total have the same upper heads and tails as $C_\Gamma$.
Therefore, $C_\Gamma$ is not primitive which contradicts our assumption.
As $C_\Gamma$ is also alternating of length at least $4$, the cycle $C_G$ thus has at least two disjoint upper path segments which in total means that $C_G$ is violating in $G_{\S,\T}$.

For the converse direction of the proof, suppose that $G_{\S,\T}$ has a violating cycle $C_G$.
As $C_G$ visits every vertex of $G_{\S,\T}$ at most once, it can not contain all upper edges or all lower edges of $G_{\S,\T}$.
Therefore, the cycle $C_G$ is also a cycle in the cone graph $\Gamma_{\sigma,\tau}$ where $\sigma\in \Sigma_\S$ and $\tau\in \Sigma_\T$ are the normal cones spanned by the rays corresponding to the upper  and lower edges of $C_G$ (we take the linear spans of the undirected edges transversed by $C_G$), respectively.
As $\Gamma_{\sigma,\tau}$ contains the transitive closure of its edges, the cycle $C_G$ yields an alternating cycle $C_\Gamma$ in $\Gamma_{\sigma,\tau}$ by replacing every path segment by one edge.
Since $C_G$ has at least two disjoint upper path segments the cycle $C_\Gamma$ has length at least $4$.
As $C_G$ visits every vertex at most once, the only vertices in $\Gamma_{\sigma,\tau}$ where the target of an upper edge meets the source of a lower edge are the ones where the upper and lower path segments are concatenated in $C_G$, i.e., the vertices where the cycle switches from the upper to the lower edges or vice versa.
This fact implies that all cycles in  $\Gamma_{\sigma,\tau}$ have their upper to lower and lower to upper switches at the same positions as the cycle $C_\Gamma$.
Thus,  the cycle $C_\Gamma$ is primitive.
Therefore, by Theorem \ref{thm:cone_criterion}, the intersection $\sigma\cap \tau$ is not a root cone and thus $\Sigma_\S\cap \Sigma_\T$ is not the normal fan of an alcoved polytope, i.e., $\S$ and $\T$ are not compatible.
\end{proof}

\subsection{Compatibility of full-dimensional simplices}
For $\S$ in $\osp(n)$ and a subset $I\subset [n]$ we denote the restriction of $\S$ to $I$  by $\S|_I$.

\begin{definition}
	We call $\S,\T \in \osp(n)$ \emph{4-interlaced} if there exist 4 distinct elements $a,b,c,d\in [n]$ such that the ordered set partitions of $\S$ and $\T$ restrict to respectively
	\[
	\S|_{a,b,c,d} =(a,b,c,d)\quad\text{ and } \quad \T|_{a,b,c,d} =(c,b,a,d).
	\]
	
	We say that $\S$ and $\T$ are \emph{6-interlaced} if there exist 6 distinct elements $a,b,c,d, e, f\in [n]$ such that the ordered set partitions of $\S$ and $\T$ restricts respectively to one of the two pairs
	\[
	\S|_{a,b,c,d,e,f}=(a,b,c,d,e,f) \quad\text{ and } \quad \T|_{a,b,c,d,e,f}=(c,d,a,b,e,f).
	\]
	\[
	\S|_{a,b,c,d,e,f}=(a,b,c,d,e,f) \quad\text{ and } \quad \T|_{a,b,c,d,e,f}=(a,d,e,b,c,f);
	\]
\end{definition}

Remarkably, these three cases completely characterize compatible nondegenerate partitions.
\begin{theorem}\label{thm:interlaced}
	Let $\S$ and $\T$ be two nondegenerate ordered set partitions. 
	Then $\S$ and $\T$ are not compatible if and only if they are $4$- or $6$-interlaced.
	
	In particular, $\S$ and $\T$ are compatible if and only if $\S_I$ and $\T_I$ are compatible for any $I$ of size at most 6.
\end{theorem}

\Cref{thm:interlaced} is a refinement of~Theorem~\hyperref[thm:B]{B} in the case of full-dimensional simplices. It will be used as the main step in the proof of general statement of Theorem~\hyperref[thm:B]{B} in the next section.

\begin{remark}\label{rmk:pattern_avoidance}
Theorem~\ref{thm:interlaced} links compatibility of full-dimensional alcoved simplices to pattern avoidance in cyclic permutations. Pattern avoidance in linear permutations is an active research topic in the past few decades \cite{kitaev2011patterns}. The study of pattern avoidance in cyclic permutations was initiated by Callan in \cite{callan2002pattern} and later studied in \cite{elizalde2021consecutive,mansour2021enumerating,li2021vincular}. Indeed, a pair of nondegenerate set partitions (or  cyclic orders) $\S$ and $\T$ defines a cyclic permutation $\pi_{\S,\T}$. Moreover, $\S$ and $\T$ are $4$--interlaced if $\pi_{\S,\T}$ contains the pattern $1432$ and $6$-interlaced if it contains the patterns  $125634$ or $145236$. Thus, by Theorem~\ref{thm:interlaced}, nondegenerate set partitions $\S$ and $\T$ are compatible if and only if $\pi_{\S,\T}$ is avoiding the above three patterns. In particular, in \cite{callan2002pattern} Callan showed that the number of cyclic permutations of length $n$ avoiding  the pattern $1432$  is equal to $2^n + 1 - 2n -\binom{n}{3}$.
\end{remark}

We begin by proving that interlaced pairs are indeed incompatible.

\begin{proposition}\label{prop:interlacing}
	Let $\S$ and $\T$ be two nondegenerate set partitions. If $\S$ and $\T$ are $4$- or $6$-interlaced, then $\S$ and $\T$ are not compatible.
\end{proposition}
\begin{proof}
	The proof proceeds by distinguishing the three types of interlacing above by finding a violating cycle in $G_{\S,\T}$ in each case.
	First assume that $\S$ and $\T$ are $4$-interlaced.
	W.l.o.g.\ after relabeling we can assume that the elements $a,b,c,d$ are the numbers $1,2,3,4$ in this order.
	Thus, $\S|_{1,2,3,4}=(1,2,3,4)$ and $\T|_{1,2,3,4}=(3,2,1,4)$.
	In this case, we find the violating cycle $1 \upto 2 \downto 3 \upto 4 \downto 1$ of length $4$ in $G_{\S, \T}$.

 Secondly, assume that $\S$ and $\T$ are $6$-interlaced of the first kind.
 We can thus again assume that  we have
 $\S|_{1,2,3,4,5,6}=(1,2,3,4,5,6)$ and $\T|_{1,2,3,4,5,6}=(3,4,1,2,5,6)$ after relabeling the elements.
 In this case, we find this violating cycle of length $6$ in $G_{\S, \T}$:
 \[1 \downto 4 \upto 5 \downto 2 \upto 3 \downto 6 \upto 1.\]

 Lastly, assume that $\S$ and $\T$ are $6$-interlaced of the second kind.
 We can thus again assume that we have
 $\S|_{1,2,3,4,5,6}=(1,2,3,4,5,6)$ and $\T|_{1,2,3,4,5,6}=(1,4,5,2,3,6)$ after relabeling the elements.
 In this case, we again find a violating cycle of length $6$ in $G_{\S, \T}$:
 \[1 \upto 2 \downto 5 \upto 6 \downto 3 \upto 4 \downto 1 .\qedhere\]
\end{proof}

The converse of this statement is the missing piece in the proof of ~\Cref{thm:interlaced}.
As a warm-up we start with the result for small $n$.
This will serve as the start of an inductive argument.

\begin{proposition}\label{prop:le_6}
	Let $\S$ and $\T$ be two nondegenerate ordered set partitions on $\left[n\right]$ with $n\le 6$.
	Assume that $\S$ and $\T$ are not compatible.
	Then they are $4$- or $6$-interlaced.
\end{proposition}
\begin{proof}
	Without loss of generality, we can assume that $\S=(1,\ldots,n)$ is the standard cyclic order. 
	It is clear that $G_{\S,\T}$ cannot contain a violating cycle for $n<4$.
	
	We thus first consider the case $n=4$.
	In this case, the only two options for violating cycles $G_{\S,\T}$ are
	\[
	1\upto 2 \downto 3 \upto 4 \downto 1 \quad \mbox{or}\quad 1\downto 2\upto 3\downto 4 \upto 1.
	\]
	In both cases the only option for $\T$ is the ordered set partition $(3,2,1,4)$ (after shifting cyclically).
	Thus $\S$ and $\T$ are $4$-interlaced.
	
	The case $n=5$ is the same as the previous case as a violating cycle on $\left[5\right]$ either contains only $4$ vertices or one path segment of length $2$ which can also be reduced to the previous case.
	
	So let us now consider the case $n=6$.
	As above, we only need to consider violating cycles in $G_{\S,\T}$ that consist of three upper path segments as $\S$ and $\T$ are $4$-interlaced otherwise.
	There are only four options for such cycles in $G_{\S,\T}$:
	\begin{eqnarray*}
		1 \upto 2\downto 3 \upto 4 \downto 5 \upto 6 \downto 1 \quad \mbox{or}\quad 1 \downto 2\upto 3 \downto 4 \upto 5 \downto 6 \upto 1\quad \mbox{or}\\
		1 \upto 2\downto 5 \upto 6 \downto 3 \upto 4 \downto 1 \quad \mbox{or}\quad 1 \downto 4\upto 5 \downto 2 \upto 3 \downto 6 \upto 1.\phantom{\quad \mbox{o}}
	\end{eqnarray*}
	The lower edges in these cycles each give three consecutive conditions on the ordered set partition $\T$.
	For instance the first cycle implies that $\T$ has the three parts $(3,2)$, $(5,4)$, and $(1,6)$ in some order.
	The only five ordered set partitions (up to cyclic shifts) that satisfy all conditions of at least one of the four cycles above are:
	\[
	(\mathbf{3},\mathbf{2},5,4,\mathbf{1},\mathbf{6})\quad (5,4,\mathbf{3},\mathbf{2},\mathbf{1},\mathbf{6})\quad (\mathbf{5},2,\mathbf{1},4,\mathbf{3},\mathbf{6})\quad (1,4,5,2,3,6)\quad (3,4,1,2,5,6).
	\]
	The last two options are $6$-interlaced by definition.
	The first three options are $4$-interlaced witnessed by the numbers displayed in bold face.
	Thus, the two ordered set partitions $\S$ and $\T$ are $4$- or $6$- interlaced in each case.
\end{proof}

\begin{proof}[Proof of~\Cref{thm:interlaced}]
	To prove the theorem we show that for any pair $\S$ and $\T$ of incompatible nondegenerate set partitions of $[n]$ with $n>6$, there exists a proper subset $I\subset [n]$ such that $\S_I$ and $\T_I$ are not compatible. Therefore, by applying this reduction iteratively, we will obtain a proper subset $I$ of size $\leq 6$ such that $\S_I$ and $\T_I$ are not compatible.
	The theorem then follows from the characterization of compatible simplices for $n\leq 6$ in~\Cref{prop:le_6}.

	To prove the existence of $I$, let us assume the following reduction.
	By Theorem~\ref{thm:violating_cycles} compatibility of $\S$ and $\T$ is equivalent to the existence of  a violating cycle $C$ in $G_{\S, \T}$. Therefore the restriction of the set partitions $\S$ and $\T$ to the vertices involved in the cycle of $C$ is still incompatible.
	Hence it is enough for us to study the case when the violating cycle $C$ passes through all vertices of  $G_{\S, \T}$.
	Moreover, we can assume that all upper and all lower path segments of $C$ are just single edges as we could otherwise restrict to the start and end vertices of these segments and obtain a violating cycle on fewer vertices; note that since the cycle is violating the vertices in the middle of such segments are met by the cycle exactly once.
	So in total the cycle $C$ visits every vertex exactly once and the edges in $C$ alternate between upper and lower vertices.
	Further without loss of generality we can assume that $\S=(1,\ldots,n)$ is the standard cyclic order on $[n]$ and $\T = (j_1,\ldots,j_n)$.

	For the cyclic order $\T=(j_1,\ldots, j_n)$ we define its $i$-th step $s_i$ to be 
	\[
	s_i =  j_{i+1}-j_{i} \mod n,
	\]
	for $i<n$ and the $n$-th step to be $j_1-j_n\mod n$.
	Let us assume that there is $i$ such that the step $s_i$  in $\T$ is not 1 or 3. We will construct a violating cycle $C'$ strictly shorter than~$C$, i.e., not passing through all vertices of $G_{\S, \T}$.
	The existence of $C'$ proves the theorem under the assumption that not all the steps in $\T$ are equal to 1 or 3.
	In the case that this assumption does not hold, the theorem follows from~\Cref{cor:13stepsinterlaced} below.
	
	Let $s_i\ne 1,3$ be the $i$-th step in $\T$. 
	We can assume that the cycle $C$ contains the edges $j_{i+1}-3 \,\upto\, j_{i+1}-2$ and $j_{i+1}-1 \,\upto\, j_{i+1}$ of $G_\S$. If not, we can consider the complementary cycle to $C$.
	Thus, the cycle $C$ consists of the concatenation of the four segments $j_{i+1}-3 \,\upto\, j_{i+1}-2$, $P_1$, $j_{i+1}-1 \,\upto\, j_{i+1}$, and $P_2$ where $P_1$ is an alternating path from $j_{i+1}-2$ to $j_{i+1}-1$ and $P_2$ is an alternating path from $ j_{i+1}$ to $j_{i+1}-3$.
	Note that both paths start and end with a lower edge.
	
	Let $C'$ be the cycle comprised of the two segments $j_{i+1}-3 \,\upto\, j_{i+1}-2 \upto j_{i+1}-1 \,\upto\, j_{i+1}$ and~$P_2$.
	It is easy to see that $C'$ is still violating.
	By removing the vertices $ j_{i+1}-2$ and  $ j_{i+1}-1$ we thus get a shorter violating cycle of the two segments $j_{i+1}-3 \,\upto\, j_{i+1}$ and $P_2$ in the graph corresponding to the restricted set partitions $\S|_{\left[n\right]\setminus\{j_{i+1}-2 \, ,\, j_{i+1}-1\}}$ and $\T|_{\left[n\right]\setminus\{j_{i+1}-2 \, ,\, j_{i+1}-1\}}$.
\end{proof}

\begin{proposition}\label{prop:13stepsinterlaced}
	Assume that $\T=(j_1,\ldots, j_n)$ is a nondegenerate ordered set partition with all steps $s_i$ equal to either $1$ or $3$. Then there are only four possible cases:
	\begin{itemize}
		\item[(1)] $s_i=1$ for all $1\le i\le n$;
		\item[(2)] $n$ is not divisible by 3 and $s_i=3$ for all $1\le i \le n$;
		\item[(3)] $n=4k+2$ and $s_{2i-1}=1$, $s_{2i}=3$ for all $1\leq i\leq 2k+1$;
		\item[(4)] $n=4k+2$ and $s_{2i-1}=3$, $s_{2i}=1$ for all $1\leq i\leq 2k+1$.
	\end{itemize}
\end{proposition}
\begin{proof}
	First notice, that for any $n$, the case (1) is realizable by the standard order and similarly case (2) is realizable if and only if $n$ is not divisible by $3$.
	
Now assume that there exist  some $i,j\in [n]$ with $s_i = 1$ and $s_j=3$. We will show that in this case, one has $s_i\cdot s_{i+1}=3$ for any $i\in [n]$. Equivalently, there are no two consecutive steps of length 1 or 3.
Without loss of generality we can assume that $j_1=1$. Let us denote by $\pi$ the cyclic permutation $(j_1,\ldots,j_n)$. In terms of the permutation $\pi$, the $i$-th step $s_i$ is given by $\pi(i)-i \mod n$. 

Assume that $s_i=s_{i+1}=1$ for some $i\in [n]$ or equivalently that $\pi(i)=i+1$ and $\pi(i+1)=i+2$.
Let $I=[l_1,l_2]$ be the largest cyclic interval containing $i,i+1$ such that $s_{k}=1$ for any $k\in I$. Since we have at least one $j$ with $s_j=3$, we have $I\ne[n]$ and in particular $(l_1-1)\notin I$. On the other hand, since $i,i+1\in I$, the length of $I$ is at least two. Therefore, 
\[
    \pi(l_1-1)=(l_1-1)+s_{l_1-1}= (l_1-1)+3 = l_1+2\leq l_2+1,
\]
which contradicts the fact that $\pi = (j_1,\ldots, j_n)$ is a cycle since $\pi(k)=k+1$ for all $k\in [l_1,l_2]$.
	
Finally, assume that $s_i=s_{i+1}=3$ for some $i\in [n]$. Then we get that $s_{i+2}=3$ as well since $i+3=\pi(i)\ne \pi(i+2)$.
Let $I=[l_1,l_2]$ be the maximal cyclic interval interval containing $i,i+1,i+2$ such that $s_{\pi(k)}=3$ for any $k\in I$. Then the step at $\pi^{-1}(l_1+2)$ has to be grater than~$3$ which contradicts the assumptions. Indeed, since $s_{l_1+1}=s_{l_1}=3$, $s_{l_1-1}=1$, and the largest possible value for $\pi^{-1}(l_1+2)$ is $l_1-2$ we obtain that $s_{\pi^{-1}(l_1+2)} \geq (l_1+2)-(l_1-2)=4$.
\end{proof}

\begin{corollary}\label{cor:13stepsinterlaced}
	Let $\S$ be the standard cyclic order on $\left[n\right]$ with $n\ge 7$ and $\T$ be another distinct cyclic order on $\left[n\right]$ with all steps $s_i$ equal to either $1$ or $3$. Then $\S$ and $\T$ are $4$- or $6$-interlaced.
\end{corollary}
\begin{proof}
	By~\Cref{prop:13stepsinterlaced} we need to consider three cases for the cyclic order $\T$ as the case with all steps equal to one yields the standard cyclic order $\S$.
	\begin{description}
		\item[Case 1] All steps of $\T$ are equal to $3$ and $n$ is not divisible by $3$.
		First assume that $n\equiv 1\mod 3$.
		Then the cyclic order $\T$ has the shape
		\[
		\mathbf{1},\mathbf{4},7,\dots \mathbf{3},6,\dots,\mathbf{2},5,\dots.
		\]
		So in this case, the numbers $1,2,3,4$ form a $4$-interlacing subsequence.
		Now assume that $n\equiv 2\mod 3$.
		Then the cyclic order $\T$ has the shape
		\[
		\mathbf{1},4,\mathbf{7},\dots 2,\mathbf{5},\dots,\mathbf{3},6,\dots.
		\]
		So in this case, the numbers $1,3,5,7$ form a $4$-interlacing subsequence.
		\item[Case 2]  $n=4k+2$ and $s_{2i-1}=1$, $s_{2i}=3$ for all $1\leq i\leq 2k+1$.
		In this case, the cyclic order $\T$ has the shape
		\[
		\mathbf{1},\mathbf{2},\mathbf{5},\mathbf{6},9,10,\dots \mathbf{3},\mathbf{4},7,8\dots.
		\]
		Hence, the numbers $1,2,3,4,5,6$ form a $6$-interlacing subsequence (of the first kind).
		\item[Case 3] $n=4k+2$ and $s_{2i-1}=3$, $s_{2i}=1$ for all $1\leq i\leq 2k+1$.
		In this last case, the cyclic order $\T$ has the shape
		\[
		\mathbf{1},\mathbf{4},\mathbf{5},8,9,\dots,\mathbf{2},\mathbf{3},\mathbf{6},7,10,\dots.
		\]
		Hence, the numbers $1,2,3,4,5,6$ form a $6$-interlacing subsequence (of the second kind).\qedhere
	\end{description}
\end{proof}

\subsection{Compatibility of lower-dimensional simplices}\label{subsec:proof_of_thmB}

Next we extend our compatibility results to include the case of degenerate ordered set partitions $\S$ and $\T$.
So in this case, their Minkowski sum $\Delta_\S+\Delta_\T$ can be not full-dimensional, or to put it differently, the common refinement of their normal fans $\Sigma_\S$ and $\Sigma_\T$ can have a nontrivial lineality space.
Now we ready to complete the proof of Theorem~\hyperref[thm:B]{B}.

\begin{proof}[Proof of Theorem~{\hyperref[thm:B]{B}}]
	If $\S$ and $\T$ are compatible then every restricted partition is compatible as well.
	For the converse, suppose $\S$ and $\T$ are not compatible.
	Our aim is now to find a subset  $I\subseteq \left[n\right]$ of size at most $6$ such that the restricted partitions $\S|_I$ and $\T|_I$ are not compatible either.
	
	By~\Cref{thm:violating_cycles} there is a violating cycle $C$ in $G_{\S,\T}$.
	Let $\widetilde{\S}$ and $\widetilde{\T}$ be two nondegenerate set partitions that refine the order of $\S$ and $\T$, respectively, and follow the directions of the undirected edges appearing in $C$.
	In other words, the graph $G_{\widetilde{\S},\widetilde{\T}}$ agrees with $G_{\S,\T}$ except that all undirected edges in $G_{\S,\T}$ are now oriented in $G_{\widetilde{\S},\widetilde{\T}}$ in such a way that all undirected edges of $G_{\S,\T}$ that appear in $C$ are oriented as prescribed by $C$.
	
	Hence, $C$ is a violating cycle in $G_{\widetilde{\S},\widetilde{\T}}$ and the two nondegenerate OSPs $\widetilde{\S}$ and $\widetilde{\T}$  are also not compatible.
	\Cref{thm:interlaced} now implies that there exists a subset $I\subset [n]$ with $|I|\leq 6$ such that $\widetilde{\S}|_I$ and $\widetilde{\T}|_I$ are not compatible either.
	Let $C_I$ be a violating cycle witnessing the latter incompatibility.
	This cycle is also a violating cycle in the restricted undirected graph $G_{\S_I,\T_I}$.
	Therefore,  $\S|_I$ and $\T|_I$ are not compatible either.
\end{proof}

We close this section by emphasizing that such a reduction property does not hold for general root subspaces.
\begin{remark}\label{rmk:non_reduction}
	For $n\ge 2$, let $L_1$ and $L_2$ be the two root spaces generated by 
	\[
	\{e_{12}, e_{34}, \dots,e_{2n-1\,2n}\}\text{ and }  
	\{e_{23},\dots,e_{2n-2\,2n-1}, e_{2n\,1}\} \text{ respectively}. 
	\]
	Then the intersection $L_1\cap L_2$ is the one-dimensional subspace generated by
	\[
	L_1\cap L_2 = \langle e_{12}+ e_{34}+ \dots+e_{2n-1\,2n}\rangle= \langle e_{23}+\dots+e_{2n-2\,2n-1}+e_{2n\,1}\rangle
	\]
	which is not a root space.
	However, the intersection $(L_1\cap I)\cap (L_2\cap I)$ for any proper subspace $I$ of $\cH_n$ is trivial, and thus a root subspace.
	Hence, the compatibility of $L_1$ and $L_2$ cannot be tested by checking it on lower-dimensional subspaces only.
\end{remark}

\section{Prominent alcoved polytopes}\label{sec:examples}

In this section we present three series of  polytopes which are shown to be alcoved using Theorem~\hyperref[thm:B]{B}: the associahedron $A_{n}$, the cyclohedron $C_n$ and the $\hat{D}_n$-polytope.
The fact that associahedra and cyclohedra are alcoved is well-known, but our techniques give a new proof.
 The conclusion that the $\hat{D}_n$-polytope is alcoved is new to our knowledge.

Following Postnikov~\cite{Postnikov}, the cyclohedron $C_n$ is typically presented as the Minkowski sum of certain faces of the coordinate simplex.
For $I\subseteq \left[n\right]$ the simplex $\Delta_I$ is the convex hull of the bases vectors $\{e_i\}_{i\in I}$.
\begin{definition}
	For $n\ge 2$, the \emph{cyclohedron} $C_n$ is the Minkowski sum
	$
	C_n = \sum_{I\subseteq \left[n\right]} \Delta_I,
	$
	where the sum runs over all cyclic intervals $I$ in $\left[n\right]$.
	Equivalently, it is the Newton polytope
	\begin{eqnarray}\label{eq: cyclo gen perm}
		C_n = \text{Newt}\left((y_1+\cdots +y_n)\prod_{i\not=j}f_{i,j}(y) \right),
	\end{eqnarray}
	where $f_{i,j}(x) = (y_{i+1}+\cdots +y_j)$ and the indices are cyclic modulo $n$.
	
	The \emph{associahedron} $A_n$ is the Minkowski sum
	$
	A_n = \sum_{I\subseteq \left[n\right]} \Delta_I,
	$
	where the sum runs over all linear intervals $I$ in $\left[n\right]$.
\end{definition}

We start by relating these Minkowski sums to the alcoved simplices discussed in~\Cref{sec:alcoved_simplices}.
The key point is that after a change of coordinates, these coordinate simplices are precisely the alcoved simplices corresponding to coarsenings of the standard ordered set partition $(1,\dots,n)$.

\begin{proposition}\label{prop:alcoved_transformation}
	Let $\varphi:\R^{n}\to \R^{n-1}$ be the linear surjection defined by 
	\[
		e_i \mapsto \sum_{j=1}^{i-1} e_j\quad \mbox{for }1\le i \le n.
	\]
	Let $I=\{i_1,\dots,i_k\}\subseteq \left[n\right]$ with $i_1<\dots <i_k$.
	The image of the coordinate simplex $\varphi(\Delta_I)$ is the shifted simplex $N_{\S_I}+\sum_{j=1}^{i_1-1}e_j$ where $\S_I$ is the following coarsening of $(1,\dots,n)$
	\[
		\S_I \coloneqq (i_1 \, i_1+1\dots i_2-1, i_2 \dots i_3-1, \dots, i_{k}\dots i_1-1).
	\]
	In particular the simplices $\varphi(\Delta_I)$ and $N_{\S_I}$ have the same normal fan in $\R^{n-1}$.
	
	This defines a bijection between the simplices $\Delta_I$ for $I\subseteq \left[n\right]$ and the OSPs that coarsen $(1,\dots,n)$.
\end{proposition}

For instance for $n=4$, the full simplex $\Delta_{\left[4\right]}$ corresponds to the OSP $(1,2,3,4)$ and the simplex $\Delta_{1,3}$ to the OSP $(1\,2,3\, 4)$.

\begin{proof}
	By the construction in Equation \eqref{eqn:newton}, the simplex $N_{\S_I}$ has the vertices
	\[
	0, \quad e_{i_1} +\dots + e_{i_2-1}, \quad e_{i_1} +\dots + e_{i_3-1},\quad \dots \quad,  e_{i_1} +\dots + e_{i_k-1}.
	\]
	On the other hand, the simplex $\varphi(\Delta_I)$ has the vertices
	\[
	e_{1} +\dots + e_{i_1-1}, \quad e_{1} +\dots + e_{i_2-1}, \quad e_{1} +\dots + e_{i_3-1},\quad \dots \quad,  e_{1} +\dots + e_{i_k-1},
	\]
	which proves that $\varphi(\Delta_I)= N_{\S_I} +\sum_{j=1}^{i_1-1}e_j$.
	For the second claim, one notes that one can reverse the above construction to obtain a corresponding coordinate simplex to a coarsening of $(1,\dots,n)$.
\end{proof}

This discussion enables us to represent both the cyclohedron and the associahedron as Minkowski sum of alcoved simplices.
\begin{corollary}\label{cor:cyclohedron}
	The cyclohedron is normally equivalent to the Minkowski sum over all coarsenings of the OSP $(1,2,\ldots, n)$ such that at most one block has more than one element.
	
	The associahedron normally equivalent to the Minkowski sum over all coarsenings of the OSP $(1,2,\ldots, n)$ such that at most one block has more than one element and $n$ is in this largest block.
\end{corollary}

Let us give some examples of Minkowski sum decompositions.  We have
\[
C_4=\Delta_{(1,2,3,4)}+\Delta_{(1,2,34)}+\Delta_{(1,23,4)}+\Delta_{(12,3,4)}+\Delta_{(2,3,41)}+\Delta_{(1,234)}+\Delta_{(123,4)}+\Delta_{(3,412)}+\Delta_{(2,341)}.
\]
Furthermore we have
\[
A_4 =\Delta_{(1,2,3,4)}+\Delta_{(1,2,34)}+\Delta_{(2,3,41)}+\Delta_{(1,234)}+\Delta_{(3,412)}+\Delta_{(2,341)}.
\]

Finally, we present the $\hat{D}_n$-polytope as follows.
\begin{definition}\label{def:dhat}
For a cyclic interval $I = [s,t]\subset [n-1]$, let  $\S_{I}$ be an ordered set partition of $[n-1]$ which is given as follows
\[
\S_{I} = ([s,t],t+1, t+2, \ldots,s-1).
\]
Now for any partition $\S_{s,t}$ of $[n-1]$ we define the partition $\hat{\S}_{s,t}$ of $[n]$, by joining the $n^\text{th}$ point to the interval $[s,t]$, resulting in
\begin{equation}\label{eq:dhat}
\hat{\S}_{s,t} = ([s,t]\cup n, t+1, \ldots, s-1).
\end{equation}
In total there are $(n-1)^2$ of such partitions and we define $\hat{D}_n$ as
$\displaystyle \sum_{r,t\in [n-1]}\Delta_{\hat{\mathbf{S}}_{r,t}}.
$
\end{definition}

The main result of this section is the following.
\begin{theorem}\label{thm: assocCycDHat}
    The associahedron $A_{n}$, the cyclohedron $C_n$ and the $\hat{D}_n$-polytope are alcoved.
\end{theorem}

\begin{proof}
    Since the  associahedron is a Minkowski summand of the cyclohedron, it is enough to show that $C_n$ and $\hat{D}_n$ are alcoved. Using Theorem~\hyperref[thm:A]{A} it is enough to show that any pair of simplices in the corresponding Minkowski sum is compatible.
    
    Let us start with $C_n$ and let $\S,\S'$ be two ordered set partitions as in \Cref{cor:cyclohedron}. Notice that for any subset $A\subset [n]$ of size $6$, the restrictions $\S|_A,\S'|_A$ again have the form of \Cref{cor:cyclohedron}. Thus by Theorem~\hyperref[thm:B]{B} the compatibility of $\S,\S'$ follows from the fact that $C_6$ is alcoved, which can be checked directly.

    The argument for $\hat{D}_n$ is similar. For $\S,\S'$ of the form \eqref{eq:dhat}, their restrictions  $\S|_A,\S'|_A$ to a subset $A\subset [n]$ of size $6$ are of the form \eqref{cor:cyclohedron} if $n\notin A$ and of the form \eqref{eq:dhat} if $n\in A$. So the compatibility of $\S,\S'$ follows from the fact that $C_6$ and $\hat{D}_6$ are alcoved.
\end{proof}

Recently, Bossinger, Telek and Tillmann-Morris introduced the so-called pellytopes, another class of interesting polytopes related to particle physics~\cite{pellytopes}.
\begin{definition}\label{def:pellytope}
	For $n\ge 1$ the \emph{pellytope} $\mathcal{P}_n$ is defined as
	\[
	\mathcal{P}_n = \text{Newt}\left(\prod_{i=1}^{n}(1+y_i)\prod_{j=1}^{n-1}(1+y_j+y_jy_{j+1})\right).
	\]
\end{definition}
The authors prove that $\mathcal{P}_n$ is a $n$-dimensional polytope with $3n-1$ facets and its number of vertices is given by Pell's number $p_n$, defined recursively by
\[
p_1 =1, \quad p_2=2, \quad \text{and} \quad p_n=2p_{n-1}+p_{n-2}.
\]
Their main result establishes that $\mathcal{P}_n$ determines a \emph{binary geometry} given by $u$-equations.

\Cref{def:pellytope} presents the pellytopes as a Minkowski sum of  simplices which are Newton polytopes.
Using~\Cref{prop:newton} we can obtain a slight transformation of the pellytopes as Minkowski sums of alcoved simplices.
Following~\Cref{prop:alcoved_transformation}, define the pellytope~$\hat{\mathcal{P}_n}$ as the Minkowski sum over all alcoved simplices $\Delta_\S$ where $\S$ is a coarsening of the standard order $(1,2,\dots,n+1)$ containing at most one block with more than one element such that $n+1$ is in this block and containing at most three blocks in total.
Using our techniques we can immediately deduce that $\hat{\mathcal{P}_n}$ is alcoved.

\begin{corollary}
The pellytope $\hat{\mathcal{P}_n}$ is alcoved.
\end{corollary}
\begin{proof}
	This follows immediately from~\Cref{thm: assocCycDHat} as $\hat{\mathcal{P}_n}$ is a subsum of the simplices that yield the associahedron.
\end{proof}

The same proof applies to the natural generalizations that arise by summing all such alcoved simplices containing up to $k$ blocks in total.
The case $k=3$ is the pellytope and the case $k=n+1$ is the associahedron.
It seems to us an interesting question for future research to determine whether these ``higher pellytopes'' are binary geometries too.

\section{Polymatroidal Blade Arrangements and the Dressian}\label{sec:blades}
In this last section, we interpret our results in the context of blades and matroidal subdivisions \cite{Ear22}.  Blade arrangements provide an elegant generalization of tropical hyperplane arrangements, in which every maximal face belong to a much larger class of polytopes than is the case for tropical hyperplane arrangements, where all faces are alcoved polytopes.
\begin{definition}[\cite{early2018honeycomb,Ear22}]
Recall that we denote by $\Sigma_\S$ the inner normal fan of the alcoved simplex $\Delta_{\mathbf{S}}$.
Let $(\S)$ be the codimension-1 skeleton of the fan $\Sigma_\S$.
Given any point $v\in \mathcal{H}_n$, let $(\mathbf{S})_v$ be the translation of $(\mathbf{S})$ from the origin to the point $v$.
We call the pair $(\S)_v$ a \emph{blade}.
 A \emph{blade arrangement} is a collection $(\mathbf{S}_1)_{v_1},\ldots, (\mathbf{S}_d)_{v_d}$ of blades on the lattice points $v_1,\ldots, v_d$.
\end{definition}
  As a special case, tropical hyperplane arrangements arise when a single cyclic order is fixed.  We are interested in allowing all possible ordered set partitions at all possible locations and placing constraints on the maximal faces of the resulting chambers.

\begin{definition}[\cite{Ear22,Cachazo2020PlanarKC}]
A blade arrangement is said to be \emph{permutohedral} provided every cell in the superposition is a generalized permutohedron. 
If in addition $v_1,\ldots, v_d$ are integer vectors, then the arrangement is called \emph{polymatroidal}.
\end{definition}
This definition prompts the main question of this section.

\begin{question}\label{question: arrangements}
	Which blade arrangements $(\mathbf{S}_1)_{v_1},\ldots, (\mathbf{S}_t)_{v_t}$ are polymatroidal?
\end{question}
A first case will be to consider the case when all $\mathbf{S}$ are cyclic orders, in which case our compatibility criterion is most powerful.  Permutohedral and especially polymatroidal blade arrangements provide a direct generalization of the notion of a tropical hyperplane arrangement.

Question \ref{question: arrangements} is motivated by challenges coming from physics, for the construction and classification of Generalized Feynman Diagrams for amplitudes arising in the CEGM framework \cite{CEGM2019}.  In this theory, a rational function, the generalized biadjoint scalar amplitude, is constructed by summing the integral Laplace transforms of the maximal cones in the tropical Grassmannian; consequently, singularities of this rational function are dual to certain realizable tropical Pl\"{u}cker vectors.  The main idea which we propose here is that \textit{any indecomposable alcoved polytope should give rise to a simple pole} of the amplitude, and the issue of compatibility of indecomposable alcoved polytopes enters when asking about its possible \textit{overlapping singularities}.

The first author gave a complete answer to Question~\ref{question: arrangements} in the case of the blade of the standard order put at different vertices of a hypersimplex $\Delta_{k,n}$.
\begin{theorem}[\cite{Ear22}]
	An arrangement of the standard blade $((1,2,\ldots, n))$ on the vertices $e_{J_1},\ldots, e_{J_d} \in \Delta_{k,n}$ induces a matroid subdivision if and only if $e_{J_i} - e_{J_j}$ alternates sign exactly twice for each $i\not=j$.
\end{theorem}
For example, $((1,2,3,4))_{e_1+e_3},((1,2,3,4))_{e_2+e_4}$
is a blade arrangement on which is not (poly) matroidal.  This is because it fully triangulates the octahedron $\Delta_{2,4}$ defined by $x_1+x_2+x_3+x_4=2$ with $0\le x_j\le 1$ into four tetrahedra with a common non-root edge.

The connection to alcoved polytopes arises via the following Lemma and Proposition.
\begin{lemma}[\cite{early2016combinatorics}]\label{lem: hypersimplex tiling}
    Fix $n\ge 2$.  The hyperplane $\mathcal{H}_n$ is tiled by integer translates of hypersimplices $\Delta_{k,n}$ for $k=1,\ldots, n-1$.  For a fixed lattice point $v\in \mathcal{H}_n$, there are exactly $\binom{n}{k}$ translated copies of $\Delta_{k,n}$ having $v$ as a vertex.
\end{lemma}

\begin{proposition}
	Let $P$ be an alcoved polytope and let $\Sigma_P$ be its normal fan.
	Choose any one of the (translated) hypersimplices $\Delta_{k,n}$ neighboring the origin provided by Lemma \ref{lem: hypersimplex tiling}; after translation, the origin defines a unique vertex $e_J \in \Delta_{k,n}$. 
	Denote by~$(\Sigma_P)_{e_J}$ the polyhedral fan obtained by translating $\Sigma_P$ from the origin to $e_J$.  Then $(\Sigma_P)_{e_J} \cap \Delta_{k,n}$ defines a matroid subdivision.  In particular, it induces a subdivision of the vertex figure, the Cartesian product $\Delta_{k-1,k}\times \Delta_{1,n-k}$.
\end{proposition}
\begin{proof}
	We need to show that no new directions beyond the existing $e_i-e_j$ are introduced after intersecting $(\Sigma_P)_{e_J}$ with $\Delta_{k,n}$.  But, given any cone $C\in (\Sigma_P)_{e_J}$ it follows that $C\cap \Delta_{k,n}$ is cut out by $2n$ additional inequalities of the form $0\le x_j\le 1$ for $j=1,\ldots, n$.  Note that some cones will have higher codimension intersections with $\Delta_{k,n}$.  This intersection does not introduce any new edge directions, hence every polytope in the subdivision is a matroid polytope.  
\end{proof}

In this language, the results of this paper give a criterion to detect polymatroidal blade arrangements for different blades all at one point.
The next question is to combine the two constructions, by varying both the cyclic orders $\mathbf{S}_j \in\osp(n)$ and the translations $v_j \in \cH_n$.
We hope to return to this question in future research.

     \begin{figure}[h!t]
	\centering
	\includegraphics[width=1\linewidth]{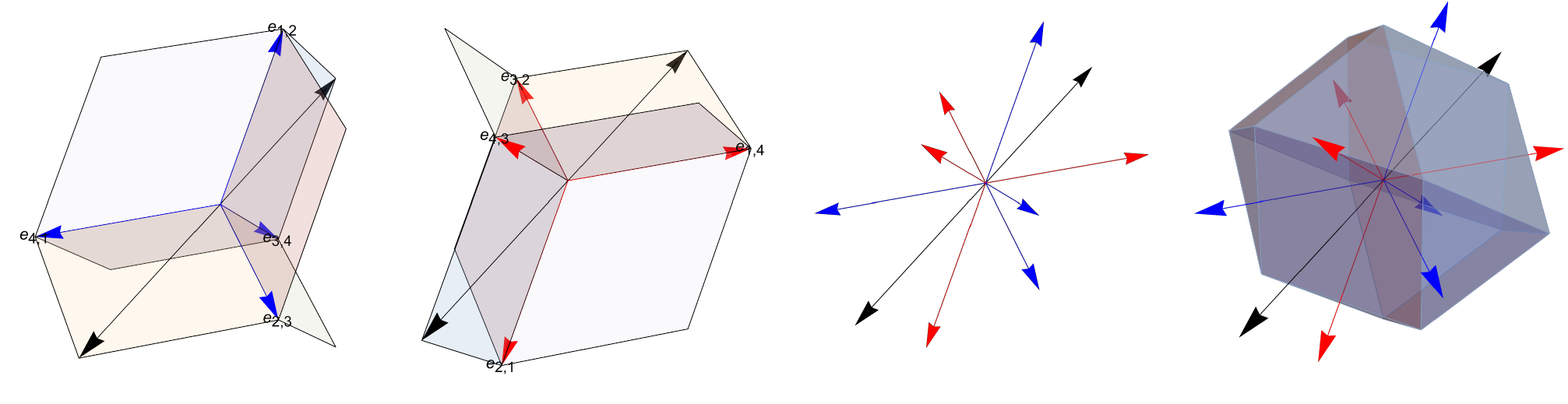}
	\caption{Left: the blade given by $((1,2,3,4))$.  Middle: the blade given by $((3,2,1,4))$.  Their common intersection is the one-dimensional subspace spanned by $e_1-e_2+e_3-e_4$; it is depicted as the black arrows.  Right: the superposition, i.e. the normal fan to the Minkowski sum of the two alcoved simplices.  Modulo a linear transformation, the Newton polytope shown is the root polytope, the convex hull of all roots $e_i-e_j$.}
	\label{fig:incompatibleBladesdim3}
\end{figure}

\begin{theorem}
    Suppose that $\mathbf{S}_1,\ldots, \mathbf{S}_d \in \osp(n)$ is a pairwise compatible collection of ordered set partitions. 
    Then the blade arrangement $(\S_1)_v,\dots (\S_d)_v$ is polymatroidal for any lattice point $v\in \cH_n$.
    Moreover, choosing $v$ to be a vertex of the hypersimplex $\Delta_{k,n}$ this induces a compatible collection of (coarsest) multi-split matroid subdivisions of $\Delta_{k,n}$, and in particular this yields a cone in the Dressian $\text{Dr}(k,n)$.
\end{theorem}
An interesting further question concerns realizability.
\begin{question}
    When are the induced subdivisions realizable, i.e., when are they induced by a realizable tropical Pl\"ucker vector?
\end{question}
According to the explicit construction of the weight vectors in \cite[Section 2.1]{Ear22}, any single blade induces a regular matroid subdivision.  But this is not the case in general for arrangements of more than one blade at different points.

For an example of a blade arrangement which induces a non-realizable tropical Pl\"ucker vector, recall that the Fano matroid polytope is a maximal face of a matroid subdivision that is induced by a non-realizable tropical Pl\"ucker vector in the Dressian~$\text{Dr}(3,7)$.  It is induced by the arrangement of seven affine hyperplanes.  In this case simply a collection of seven compatible two-splits of $\Delta_{3,7}$, induced by (the sum of) the seven tropical Pl\"ucker vectors
$$e_{123},e_{145},e_{167},e_{246},e_{257},e_{347},e_{356} \in \mathbb{R}^{\binom{7}{3}}.$$
As is generally true, such a subdivision can be induced in many different ways by blade arrangements on the vertices of $\Delta_{3,7}$.

\subsection*{Acknowledgments}

The authors would like to thank Federico Ardila, Christian Haase, Thomas Lam, Alex Postnikov, Raman Sanyal, and Benjamin Schr\"{o}ter for fruitful discussions.
The main part of this research was carried out while the authors stayed as Oberwolfach Research Fellows at the Oberwolfach Research Institute for Mathematics.

N.E. was partially supported by the European Union (ERC, UNIVERSE PLUS, 101118787).  Views and opinions expressed are however those of the author(s) only and do not necessarily reflect those of the European Union or the European Research Council Executive Agency. Neither the European Union nor the granting authority can be held responsible for them.
L.K. was partially supported by the DFG -- SFB-TRR 358/1 2023 -- 491392403 and SPP 2458 -- 539866293.
Part of the research was carried out while L.K. was an Erik Ellentuck fellow at the Institute for Advanced Study.
L.M. was partially supported by  DFG SPP 2458 -- 539974215 as well as SNSF grant -- 200021E\_224099.

\renewcommand*{\bibfont}{\small}
\printbibliography

\end{document}